\documentclass[12pt,x11names]{article}

\usepackage{natbib}
 \bibpunct[, ]{(}{)}{,}{a}{}{,}%
\usepackage{geometry,amsmath,amssymb,amsthm,graphicx,epsfig}
\usepackage{newfloat}
\usepackage{subcaption}
\usepackage{capt-of}
\usepackage{enumerate,dcolumn,latexsym,epstopdf,color}
\usepackage{enumitem}
\usepackage{multicol,multirow}
\usepackage{lscape}
\usepackage{algorithm}
\usepackage{tabularx}
\usepackage{authblk}

\newtheorem{theorem}{Theorem} 
\newtheorem{corollary}{Corollary} 
\newtheorem{proposition}{Proposition} 
\newtheorem{lemma}{Lemma} 
\theoremstyle{definition}
\newtheorem{remark}{Remark} 
\newtheorem{assumption}{Assumption}{{\Alph{assumption}}}

\newcommand{\paren}[1]{\left\{ {#1} \right\}} 
\newcommand{\real}{\mathbb{R}} 
\newcommand{\integ}{\mathbb{Z}} 
\newcommand{\inter}{\mathrm{int}} 
\newcommand{\norm}[1]{\left\Vert {#1} \right\Vert} 
\newcommand{\conv}{\mathrm{conv}} 

\newcommand{\ie}{{\em i.e.}, }
\newcommand{\eg}{{\em e.g.}, }

\newcommand{\bzero}{{\bf 0}}
\newcommand{\bone}{{\bf e}}
\newcommand{\ba}{{\bf a}}
\newcommand{\bb}{{\bf b}}
\newcommand{\bc}{{\bf c}}
\newcommand{\bd}{{\bf d}}

\newcommand{\bu}{{\bf u}}
\newcommand{\bv}{{\bf v}}
\newcommand{\bw}{{\bf w}}
\newcommand{\bx}{{\bf x}}
\newcommand{\by}{{\bf y}}
\newcommand{\bz}{{\bf z}}
\newcommand{\bba}{{\bf A}}
\newcommand{\bbb}{{\bf B}}
\newcommand{\bbc}{{\bf C}}
\newcommand{\bbd}{{\bf D}}

\newcommand{\bbi}{{\bf I}}

\newcommand{\bbz}{{\bf Z}}
\newcommand{\balpha}{{\boldsymbol \alpha}}
\newcommand{\bbeta}{{\boldsymbol \beta}}
\newcommand{\bdelta}{{\boldsymbol \delta}}

\newcommand{\blambda}{{\boldsymbol \lambda}}
\DeclareMathOperator*{\argmax}{arg\,max}
\DeclareMathOperator*{\argmin}{arg\,min}

\begin{document}



\title{A Scalable Algorithm for Two-Stage Adaptive Linear Optimization}

\author[1]{Dimitris Bertsimas\thanks{dberts@mit.edu}}
\author[1]{Shimrit Shtern\thanks{sshtern@mit.edu}}
\affil[1]{Operations Research Center,  Massachusetts Institute of Technology}
\date{}

\maketitle

\abstract{%
	The column-and-constraint generation (CCG) method was introduced by \citet{Zeng2013} for solving two-stage adaptive optimization. We found that the CCG method is quite scalable, but sometimes, and in some applications often, produces infeasible first-stage solutions, even though the problem is feasible. In this research, we extend the CCG method in a way that (a) maintains scalability and (b) always produces feasible first-stage decisions if they exist. We compare our method to several recently proposed methods and find that it reaches high accuracies faster and solves significantly larger problems.
}%



\section{Introduction}
The \emph{robust optimization} (RO) methodology is an approach which deals with uncertainty in the parameters of an optimization problem. It differs from probabilistic approaches to uncertainty, such as stochastic programming, by the fact that it assumes knowledge of an uncertainty set rather than specific probability distributions. A full description of the methodology and its various applications can be found in \citep{BenTal2009,Bertsimas2011}. The RO approach conducts worst case analysis resulting in semi-infinite problems, since each constraint must be satisfied for each point in the continuous uncertainty set. For a wide range of problems, assuming that all the decision variables are ``here and now" decisions made without knowledge of the uncertainty realization and for many types of uncertainty sets, finding a solution is both theoretically and practically tractable. Such a solution may be obtained by either using the \emph{robust counterpart}, which is a deterministic finite reformulation, or by generating constraints only as needed until the solution
is optimal and feasible with respect to the uncertain constraints \citep{Fischetti2012}.

A more complicated case is when the decisions are partly ``here and now" and partly adjustable ``wait and see", for which the variables are functions of the unknown realization of the uncertain data. Problems of this kind may arise in dynamic systems with planning variables, such as the decision whether or not to build a facility or determining its capacity, and control variables, such as allocation of resources or actual production. 
If no restriction is placed on the structure of the adjustable function, the problem becomes computationally hard, both theoretically  \citep{BenTal2004} and practically. Therefore, in many cases the adjustable functions are restricted to be affine, to ensure tractability as suggested by \citet{BenTal2004}, and this type of adjustability has been applied to control theory \citep{Goulart2006} and supply chain management \citep{BenTal2005,BenTal2009}. However, this restricted adjustability may result in suboptimal decisions in the non-adjustable variables as well. 
An alternative approach is the adaptive optimization (AO) approach, which aims to provide approximations of the true adaptive function and to quantify, if possible, how good these approximations are. This approach was primarily investigated in the context of two-stage models, although some recent work also deals with extensions for the multi-stage case. In this paper, we focus on the following linear two-stage problem
\begin{equation}\label{eq:2stage_robust}
\begin{aligned}
\min_{\bx\in X}\max_{\bu\in U} \min_{\by\in \real^M:\bba\bx+\bbb\by+\bbc\bu\geq \bc, \by\geq 0} \ba^T\bx+\bb^T\by,
\end{aligned}
\end{equation}
where $\ba\in \real^n$, $\bb\in\real^m$, $\bba\in\real^{r\times n}$, $\bbb\in \real^{r\times m}$, $\bbc\in \real^{r\times l}$, $\bc\in \real^{r}$ are the problem's parameters, $\bx\in X\subseteq \real^{n_1}\times \integ^{n_2}$, $n=n_1+n_2$ is the first-stage decision variable, $\by\in\real^m$ is the continuous second-stage decision, and $\bu$ is the uncertain parameter which lies in a convex and compact set $U$. 
In particular, we treat the case where $U$ is a polytope of the form
\begin{equation}
\label{eq:Define_U}
\paren{\bu\in \real^{l}:\bbd\bu\leq \bd, \bu\geq\bzero},
\end{equation} such that $\bbd\in \real^{d\times l}$ and $\bd\in \real^d$. 
Although this two-stage model is the simplest model where decision values can depend on the uncertainty, it still encompasses a large variety of problems, such as network problems \citep{AtamturkZhang2007,Gabrel2014}, portfolio optimization \citep{Takeda2007}, and unit commitment problems \citep{Bertsimas2013,Zhao2013}. 

We first assume that Problem \eqref{eq:2stage_robust} is in fact bounded. This is usually the case when the objective function depicts cost/profit or some physical quantity such as energy.
We define the set of feasible second-stage decisions as
\begin{equation}
Y(\bx,\bu)=\{\by:\bba\bx+\bbb\by+\bbc\bu\geq \bc, \by\geq \bzero\}.
\end{equation} 
Problem~\eqref{eq:2stage_robust} satisfies one of the following assumptions regarding this set.
\begin{assumption}[Relatively complete
recourse]\label{ass:full recourse}
For any $\bx\in X$ and any $\bu\in U$ the set  $Y(\bx,\bu)\neq \emptyset$.
\end{assumption}
\begin{assumption}[Feasibility]\label{ass:feasibility}
There exists $\bx\in X$ such that for any $\bu\in U$ the set  $Y(\bx,\bu)\neq \emptyset$.
\end{assumption}
Assumption~\ref{ass:full recourse} implies that any first-stage decision is in fact feasible for the two-stage problem, while Assumption~\ref{ass:feasibility} only assumes existence of at least one such solution. Thus, it is clear that Assumption~\ref{ass:feasibility} is more general than Assumption~\ref{ass:full recourse}.

Specific approaches to solving this two-stage problem via AO include Benders decomposition type methods \citep{benders1962}. Benders decomposition methods are iterative methods that alternate between finding a lower bound and an upper bound of the objective function until both bounds converge.
In order to obtain the upper bound for a fixed first-stage decision, we are required to solve the adversarial problem,  that is, to find an uncertainty realization that results in the worst objective function value (this value can be thought of as infinity if the first-stage decision does not yield a feasible second-stage decision).
Finding this upper bound is the more challenging part, since the upper bound problem is nonconvex, and Benders decomposition methods differ from each other in the way they approximate this upper bound. 
In particular, both linearization \citep{Bertsimas2013} and a complementarity based mixed-integer optimization (MIO) approach, called column-and-constraint generation (CCG) \citep{Zeng2013}, were suggested as possible solutions. 
In particular the CCG method guarantees reaching the optimal solution for Problem~\eqref{eq:2stage_robust} in a finite number of steps. \citet{Zeng2013} acknowledge that for the case where Assumption~\ref{ass:feasibility} holds (but Assumption~\ref{ass:full recourse} does not) the CCG method requires a feasibility oracle, but they do not provide such an oracle for the general problem. However, the CCG method is easy to implement and practically scalable for the case where Assumption~\ref{ass:full recourse} does hold, and therefore it is desirable to extend it to the case where only Assumption~\ref{ass:feasibility} holds. 

More general models of AO, that can also be applied to multi-stage problems, include: finding a near optimal solution using cutting planes and partitioning schemes \citep{BertsimasGeorghiou2015}, and finding piece-wise affine strategies by partitioning the uncertainty set using active constraints and uncertainties \citep{BertsimasDunning2014,PostekDenHertog2014}. \citet{BertsimasDunning2014} showed that the near-optimal scheme suggested by \citet{BertsimasGeorghiou2015} is not scalable, and that the partitioning schemes, while more scalable, might not reach the optimal solution in reasonable time.

To illustrate the underlying problems with the current methods, we tested the aforementioned algorithms on the capacitated network lot-sizing problem, the location-transportation problem, and unit commitment problem IEEE-14 bus example (see Section~\ref{sec:NumericalResults} for details). The summary for the AMIO partitioning model suggested by \citet{BertsimasDunning2014}, as well as the CCG model suggested by \citet{Zeng2013} are given in Table~\ref{Tbl:CompareMethods}. The table presents the percentage of instances that obtained a feasible solution, the percent of instances that terminated, \ie their upper and lower bound achieved the required $0.1\%$ gap at the specified time limit, the average and standard deviation of the time it took those instances to terminate, and the average optimality gap for the instances that were not terminated.  We can see that while the CCG always terminates before the time limit, in more than $30\%$ of the cases it results in an infeasible (and hence a non-optimal) solution. In contrast, the AMIO usually achieves a feasible solution, but takes a long time to obtain a high accuracy solution. For example, in all three unit commitment scenarios the CCG did not obtain a feasible solution while the AMIO was unable to obtain any solution (hence the missing values in the table).

\renewcommand{\arraystretch}{0.7}
\begin{table}[h]
\caption{Comparison between existing algorithms.}\label{Tbl:CompareMethods}
\begin{center}
\begin{tabular}{|l|l|c|c|c|}
\hline
Problem Type & Algorithm & \% Feasible&  \% Instances  terminated& Mean (std) time  for\\
&& Instances &(Optimality gap) & optimal  instances (sec)\\
\hline
{Location-}  & CCG & 69\% & 100\% & 27.22 (15.78)\\
Transportation& AMIO & 100\% & 46\% (0.5\%) & 449.4 (285.06)\\
\hline
{Capacitated}  & CCG & 66\% & 100\% &  10.5 (2.96)\\
 Lot-Sizing & AMIO & 100\% & 97\% (0.4\%) & 151.22 (158.4)\\
\hline
Unit   & CCG & 0/3 & 3/3 & -\\
Commitment& AMIO & - & 0/3 & -\\
\hline
\end{tabular}
\end{center}
\end{table}

In the case $U$ is a polyhedral set,
\citet{Bertsimas2015} showed that Problem~\eqref{eq:2stage_robust} has a two-stage robust dual formulation. This dual formulation has the desirable property that the optimal second-stage decisions do not depend on the first-stage one, but rather only on the dual uncertainty. The authors also show that using both primal and dual formulations with affine decision rules leads to a tighter lower bound on the optimal objective value of Problem~\eqref{eq:2stage_robust} than the one obtained by using only the primal formulation. 

In this paper, we suggest combining the ideas of CCG \citep{Zeng2013}, the dual representation of the problem \citep{Bertsimas2015}, and AMIO \citep{BertsimasDunning2014} to construct an algorithm which
\begin{enumerate}
\item Extends the result of \citet{Zeng2013} to Problem \eqref{eq:2stage_robust} under Assumption~\ref{ass:feasibility} (rather than Assumption~\ref{ass:full recourse}), \it{i.e.}, it guarantees convergence to an optimal solution in a finite number of steps.
\item Is scalable even for the case in which the first-stage decision has integer components.
\item Has superior performance to both CCG and AMIO for several numerical examples.
\end{enumerate}

{\noindent \bf Paper structure:}
In Section \ref{sec:CCG}, we present the CCG algorithm suggested in \citep{Zeng2013}, discuss its implementation for the case of Problem~\eqref{eq:2stage_robust}, and show its convergence properties for the case Assumption~\ref{ass:full recourse} holds. We conclude Section~\ref{sec:CCG} by presenting a natural extension of the CCG algorithm to the case where only Assumption~\ref{ass:feasibility} holds, and discuss why this extension is not practical. In Section
\ref{sec:DDBD}, we present our new \emph{Duality Driven Bender Decomposition} (DDBD) method for solving Problem~\eqref{eq:2stage_robust} satisfying Assumption~\ref{ass:feasibility}, proving convergence in a finite number of iterations. In Sections \ref{sec:F1}-\ref{sec:F2}, we describe the two main sub-algorithms used by our method.
Finally, in Section~\ref{sec:NumericalResults}, we compare the existing AMIO and CCG methods with our new DDBD method, using numerical examples for the location-transportation problem, the capacitated network lot-sizing problem,  and the unit commitment problem.

{\noindent \bf Notations:}
We use the following notation throughout the  paper. A column vector is denoted by boldface lowercase $\bx$ and a matrix by boldface uppercase $\bba$.  For a full rank matrix $\bba$, $\bba^{-1}$ denotes its inverse. We use the notation $\norm{\bx}$ to denote a norm of a vector and $\bx^T$, $\bba^T$ to denote a vector and matrix transpose. An element of the vector will be denoted by $x_i$, an element of the matrix by $A_{ij}$, and the matrix $i$th row by $\bba_i$.  Superscript will be used to denote an element of a set, \eg $\bx^p\in X$ is a vector which is associated with the $p$ index of the set $X$, it may also be used to denote an iterator. The notation $|a|$ is used to denote the absolute value, if $a$ is a scalar, and the cardinality of $a$ if it is a set. The vectors $\bzero$ and $\bone$ are vectors of all zeros or all ones, respectively, and the vector $\bone_i$ denotes the $i$th vector of the standard basis. We denote the identity matrix by $\bbi$.

\section{Column-and-Constraint Generation (CCG)}\label{sec:CCG}
In this section, we present the CCG algorithm suggested by \citet{Zeng2013} and prove that under Assumption~\ref{ass:full recourse} it will converge in a finite number of steps to an optimal solution of Problem \eqref{eq:2stage_robust}, and that if the assumption does not hold it might converge to an infeasible first-stage decision.
\subsection{CCG Algorithm Description}
The CCG algorithm is a general iterative method to find an optimal solution of two-stage problems. Specifically, it alternates between finding a lower bound $LB$ on the objective value and the corresponding first-stage decision $\bx$ using a finite number of uncertainty realizations, and an upper bound for the second-stage objective given $\bx$. 

Before we discuss the full algorithm, we present the optimization problem that will be used to obtain the lower bound in the first part of the algorithm.
Given $\bx$ for any $V=\{\bu^0,\ldots,\bu^{k}\}\subseteq U$, a lower bound for the second-stage cost of  Problem~\eqref{eq:2stage_robust} can be obtained using the following optimization problem
\begin{equation}\label{Prob:LowerBound_SecondStage}\underline{Z}(\bx,V)=\min_{\theta,\{\by^s\}_{0\leq s\leq k}} \paren{\ba^T\bx+\theta:\;\bb^T\by^s\leq \theta,\;\bba\bx+\bbb\by^s+\bbc\bu^s\geq \bc,\;\by^s\geq \bzero,\; s=0,\ldots,k}.\end{equation}

Notice that given a first-stage decision $\bx$ and an uncertainty parameter $\bu$, the second-stage cost is given by \begin{equation}\label{eq:SecondStage_Decision}\underline{Z}(\bx,\{\bu\})=\min_{\by\in\real^m_+,\bba\bx+\bbb\by+\bbc\bu\geq \bc} \bb^T\by\equiv \min_{\by\in Y(\bx,\bu)} \bb^T\by \end{equation}

Given a first-stage decision $\bx$, an upper bound on the second-stage cost is given by
\begin{equation}\label{eq:SecondStage_UpperBound}\overline{Z}(\bx)=\max_{\bu\in U} \underline{Z}(\bx,\{\bu\})\equiv \ba^T\bx+ \max_{\bu\in U} \min_{\by\in \real^M:\bba\bx+\bbb\by+\bbc\bu\geq \bc, \by\geq 0} \bb^T\by,\end{equation}
and the uncertainty realization that results in this cost is given by
\begin{equation}\label{eq:SecondStage_WorstCaseRealization}\overline{\bu}(\bx)\in \argmax_{\bu\in U}\underline{Z}(\bx,\{\bu\}).\end{equation}
Given these definitions, the CCG algorithm is presented as Algorithm~\ref{alg:CCG}.

\begin{algorithm}[h]\caption{Column-and-Constraint Generation (CCG)}\label{alg:CCG}
		\begin{itemize}
       		\item[$\rm{(1)}$] {\bf Input:} $X,\bba,\bbb,\bbc,\ba,\bb,\bc,\bbd,\bd, \epsilon>0,\bu^0\in U.$
            \item[$\rm{(2)}$] {\bf Initialize:} ${UB}^0=\infty$, ${LB}^0=-\infty$, $V^0=\{\bu^0\}, k=0$
            \item[$\rm{(3)}$] While $\frac{{UB}^k-{LB}^k}{\max(\min(|{LB}^k|,|{UB}^k|),1)}>\epsilon$
            \begin{itemize}
            \item[$\rm{(a)}$] Update $k\leftarrow k+1$.
            \item[$\rm{(b)}$] Compute $\bx^k=\argmin_{\bx\in X} \underline{Z}(\bx,V^{k-1})$ and the correcponding lower bound value ${LB}^k=\underline{Z}(\bx^k,V^{k-1})$.
            \item[$\rm{(c)}$] Compute $\bu^{k}=\overline{\bu}(\bx^k)$ and the corresponding upper bound ${UB}^k=\overline{Z}(\bx^k)\equiv\underline{Z}(\bx^k,\{\bu^k\})$.
            \item[$\rm{(d)}$] Update $V^{k}\leftarrow V^{k-1}\bigcup\{\bu^{k}\}$.           
            \end{itemize}
            \item[$\rm{(4)}$] Return $\bx^k$, ${UB}^{k}$.
        \end{itemize} 
\vskip 2mm
\end{algorithm}

The algorithm is general in the sense that if $\bx$ is not feasible, \ie there exists a $\bu\in U$ such that $Y(\bx,\bu)=\emptyset$, then at Step 3c, $\bu^k$ would return such a $\bu$ and $UB=\infty$. Finding $\overline{Z}(\bx)$ and $\overline{\bu}(\bx)$ in the case that Assumption~\ref{ass:full recourse} holds requires only an optimality oracle for the problem $\max_{\bu\in U}\underline{Z}(\bx,\{\bu\})$, since for all $\bx\in X$ we have that $\overline{Z}(\bx)<\infty$. However, if the less restrictive Assumption~\ref{ass:feasibility} holds true, then the algorithm additionally requires a feasibility oracle, which determines if $\bx$ is feasible. In general, even for a feasible $\bx$ (such that $Y(\bx,\bu)\neq \emptyset$ for any $\bu\in U$), computing $\overline{Z}(\bx)$ and $\overline{\bu}(\bx)$ is NP-Hard. \citet{Zeng2013} suggest a general optimality oracle based on complementary slackness, however, they do not present a general feasibility oracle.
Proposition~\ref{prop:CCG_equiv} and its proof (which we add for the sake of completeness) presents their suggested optimality oracle.
\begin{proposition}\label{prop:CCG_equiv}
Let
\begin{equation}\label{eq:SecondStage_WorstCaseCost}
\begin{aligned}[t]
&\tilde{Z}(\bx)=\ba^T\bx+\max_{\bu\in U,\by\in\real^m_+,\bw\in \real_+^r} \bb^T\by\\
&\begin{array}{crl}\qquad\qquad\qquad \textnormal{s.t.} & \bba\bx+\bbb\by+\bbc\bu&\geq \bc\\
& \bbb^T\bw&\leq \bb\\
& \bw^T(\bba\bx+\bbb\by+\bbc\bu-\bc)&=0\\
& \by^T(\bb-\bbb^T\bw)&=0,\end{array}
\end{aligned}
\end{equation}
with a corresponding maximizer $\tilde{\bu}(\bx)$ (not necessarily unique). If $\bx\in X$ is a feasible solution to Problem \eqref{eq:2stage_robust}, \ie $\overline{Z}(\bx)<\infty$, then  Problems  \eqref{eq:SecondStage_UpperBound} and \eqref{eq:SecondStage_WorstCaseCost} are equivalent, \ie
$\overline{Z}(\bx)=\tilde{Z}(\bx)$ and they have the same optimal solution set. 
Moreover, the two equality constraints can be reformulated as $r+m$ SOS-1 constraints, or using additional $m+r$ binary variables 
as
\begin{equation}\label{eq:SecondStage_WorstCaseMIO}
\begin{aligned}[t]
&\max_{\substack{\bu\in U,\by\in\real^m_+,\bw\in \real_+^r,\\
		\balpha\in\{0,1\}^r,\bbeta\in\{0,1\}^m}} \bb^T\by\\
&\begin{array}{crl}\qquad \textnormal{s.t.} & \bba\bx+\bbb\by+\bbc\bu&\geq \bc\\
& \bbb^T\bw&\leq \bb\\
& \bw&\leq \mathcal{M}\balpha\\
& \bba\bx+\bbb\by+\bbc\bu-\bc&\leq \mathcal{M}(1-\balpha)\\
& \by&\leq \mathcal{M}\bbeta\\
&\bb-\bbb^T\bw&\leq\mathcal{M}(1-\bbeta),\end{array}
\end{aligned}
\end{equation}
where $\mathcal{M}$ is a sufficiently large number.
\end{proposition}
\begin{proof}{Proof.}
The equivalence between \eqref{eq:SecondStage_WorstCaseCost} and \eqref{eq:SecondStage_WorstCaseMIO} stems from a known conversion between complementarity and Big-M type constraints. We therefore focus on proving the equivalence between \eqref{eq:SecondStage_UpperBound} and \eqref{eq:SecondStage_WorstCaseCost}.
 
Since Problem \eqref{eq:2stage_robust} is bounded, then for every $\bx$ there exists $\bu$ such that $\underline{Z}(\bx,\{\bu\})$ is bounded from below. Therefore, denoting $U(\bx)=\{\bu\in U:\underline{Z}(\bx,\{\bu\})>-\infty\}$ we have that
$$\max_{\bu\in U}\underline{Z}(\bx,\{\bu\})=\max_{\bu\in U(\bx)}\underline{Z}(\bx,\{\bu\}).$$
Since $\bx\in X$ is feasible for the problem, then $Y(\bx,\bu)\neq \emptyset$ for any $\bu\in U$ and the second-stage Problem \eqref{eq:SecondStage_Decision} is always feasible, \ie ${\underline{Z}(\bx,\{\bu\})<\infty}$.

Since Problem \eqref{eq:SecondStage_Decision} is feasible and bounded, according to linear duality theory  the following dual problem is also feasible and bounded
\begin{equation*}
\max_{\substack{\bw\in \real_+^r:\bbb^T\bw\leq \bb}} (\bba\bx-\bbc\bu)^T\bw.
\end{equation*}
Moreover, any \emph{feasible} primal-dual pair $(\by,\bw)$ satisfying the complementary-slackness conditions 
\begin{align}
\bw^T(\bba\bx+\bbb\by+\bbc\bu-\bc)&=0,\label{eq:CS1}\\
\by^T(\bb-\bbb^T\bw)&=0\label{eq:CS2},
\end{align}
is an \emph{optimal} primal-dual pair. We will denote the set of pairs $(\by,\bw)$ that satisfy \eqref{eq:CS1}-\eqref{eq:CS2} as $C(\bx,\bu)$.
Thus, we have that
\begin{equation}\label{eq:SecondStage_CSequiv}\underline{Z}(\bx,\{\bu\})=\ba^T\bx+\bb^T\by\quad \forall \by\in Y(\bx,\bu): \exists \bw\in\real^r_+,\bbb^T\bw\leq \bb,(\by,\bw)\in C(\bx,\bu).\end{equation}
Since the RHS of Equation \eqref{eq:SecondStage_CSequiv} is fixed for any $\by$ satisfying the constraints, we have
\begin{equation}\label{eq:SecondStage_CSequiv2}\underline{Z}(\bx,\{\bu\})=\ba^T\bx+\max_{\substack{\by\in Y(\bx,\bu),\bw\in\real^r_+,\\
\bbb^T\bw\leq \bb,(\by,\bw)\in C(\bx,\bu)}}\bb^T\by.\end{equation}
Notice that $\bu$ for which $\underline{Z}(\bx,\{\bu\})$ is unbounded will result in the RHS of \eqref{eq:SecondStage_CSequiv2} being infeasible, thus the equality holds even for $\bu\notin U(\bx)$. 
Maximizing both sides of equation \eqref{eq:SecondStage_CSequiv2} over all $\bu\in U$ leads to the equivalence between \eqref{eq:SecondStage_UpperBound} and \eqref{eq:SecondStage_WorstCaseCost}. 
\end{proof}

Next, in Proposition~\ref{prop:CCG_converge} we show convergence of the CCG to the optimal first-stage decision in a finite number of steps for a polyhedral $U$, provided Assumption~\ref{ass:full recourse} holds. For the proof of this result we will first need the following auxiliary lemma.
\begin{lemma}\label{lemma:Maximizer_ExtremePoint}
Let $\bx\in X$ be a feasible solution to Problem~\eqref{eq:2stage_robust}. If $U$ is a compact set, then there always exists $\overline{\bu}(\bx)$ which is an extreme point of $U$.
\end{lemma}
\begin{proof}{Proof.}
Since $\bx\in X$ is a feasible first-stage decision then $\overline{Z}(\bx)<\infty$. The function $\underline{Z}(\bx,\{\bu\})$ is continuous in $\bu$ and $U$ is compact; therefore, the maximizer $\overline{\bu}(\bx)$ defined in \eqref{eq:SecondStage_WorstCaseRealization} is attained.
Let us assume to the contrary, that all maximizers $\overline{\bu}(\bx)$ are not an extreme point of $U$. Since $Y(\bx,\overline{\bu}(\bx) )$ is nonempty and the objective function is bounded, by strong duality we have that
$$\overline{Z}(\bx)=\ba^T\bx+\min_{\by\in Y(\bx,\overline{\bu}(\bx))} \bb^T\by=\ba^T\bx+\max_{\bw\in\real^r_+:\bbb^T\bw\leq \bb} (\bc-\bba\bx-\bbc\overline{\bu}(\bx))^T\bw,$$
and the RHS maximum is attained.
Let $\bw^*\in \argmax_{\bw\in\real^r_+:\bbb^T\bw\leq \bb} (\bc-\bba\bx-\bbc\overline{\bu}(\bx))^T\bw$. We have that
\begin{equation}\label{eq:thm2_eq1} (\bc-\bba\bx-\bbc\overline{\bu}(\bx))^T\bw^*\leq\max_{\bu\in U} (\bc-\bba\bx-\bbc\bu)^T\bw^*.\end{equation}
Since the RHS of \eqref{eq:thm2_eq1} is a maximization of a linear function over a convex compact domain, and it is trivially feasible ($U\neq\emptyset$) and bounded (since $\overline{Z}(\bx)<\infty$), there must exist a maximizer $\bu^*\in  \argmax_{\bu\in U}(\bc-\bba\bx-\bbc\bu)^T\bw^*$ that is an extreme point of $U$. By definition 
$$\overline{Z}(\bx)\geq \ba^T\bx+(\bc-\bba\bx-\bbc\bu^*)^T\bw^*.$$ Thus, choosing $\overline{\bu}(\bx)=\bu^*$ leads to the desired contradiction.
\end{proof}
\begin{proposition}[Extension of {\citep[Proposition 2]{Zeng2013}}]\label{prop:CCG_converge}
Let Assumption~\ref{ass:full recourse} hold, and let $\tilde{Z}(\cdot)$ and $\tilde{\bu}(\cdot)$ (defined in \eqref{eq:SecondStage_WorstCaseCost}) be used  in Algorithm~\ref{alg:CCG} instead of $\overline{Z}(\cdot)$ and $\overline{\bu}(\cdot)$, respectively. If $U$ is convex, and both $U$ and $X$ are compact, then any limit point of the sequence $\{\bx^k\}_{k\in\mathbb{N}}$ is an optimal solution of Problem~\eqref{eq:2stage_robust}.  If alternatively $U$ is a polytope, then $\overline{\bu}(\bx)$ can always be chosen to be a vertex of $U$, and Algorithm~\ref{alg:CCG} terminates in a finite number of steps.
\end{proposition}
The proof of the proposition is given in Appendix~\ref{appx:convergence proof}.

Notice that the optimality oracle given in Proposition~\ref{prop:CCG_equiv} is only valid for feasible first-stage decisions. Moreover, Proposition~\ref{prop:CCG_converge} relies on the fact that $\bu^k$ generated by the algorithm are vertices of the polytope $U$. Proposition \ref{prop:CCG_Limitation} shows that applying the optimality oracle in Proposition~\ref{prop:CCG_equiv} for cases where Assumption~\ref{ass:full recourse} does not hold may result in an underestimation of $\overline{Z}(\bx^k)$ and/or a $\bu^k$ which is not a vertex of $U$. As a direct result Algorithm~\ref{alg:CCG} might not converge in a finite number of steps or might converge to an infeasible solution.
\begin{proposition}\label{prop:CCG_Limitation}
Let $\bx^k$ be a first-stage decision generated at iteration $k$ of Algorithm~\ref{alg:CCG} using $\tilde{Z}(\cdot)$ instead of $\overline{Z}(\cdot)$ and $\tilde{\bu}(\cdot)$ instead of $\overline{\bu}(\cdot)$. If $\bx^k$ is not feasible ($\overline{Z}(\bx^k)=\infty$), then $\tilde{Z}(\bx^k)<\overline{Z}(\bx^k)$.
Moreover, if $U$ is a polytope, then there does not necessarily exist a $\tilde{\bu}(\bx^k)$ that is a vertex of $U$.
\end{proposition}
\begin{proof}{Proof.}
Notice that since $\bx^k$ is generated by Algorithm~\ref{alg:CCG}, it must satisfy that
$\underline{Z}(\bx^k,V^{k-1})<\infty$. Moreover, for any $\bu\in U$ we have that
\begin{equation}\label{eq:CS_breakdown}
\max_{\substack{\by\in Y(\bx^k,\bu),\bw\in\real^r_+,\\
\bbb^T\bw\leq \bb,(\by,\bw)\in C(\bx^k,\bu)}}\bb^T\by=\begin{cases}
\underline{Z}(\bx^k,\{\bu\})-\ba^T\bx,& \text{if }Y(\bx^k,\bu)\neq\emptyset\; (\underline{Z}(\bx^k,\{\bu\})<\infty),\\
-\infty, & \text{otherwise}. 
\end{cases}
\end{equation}
Denoting $U(\bx)=\{\bu\in U:Y(\bx,\bu)\neq \emptyset\}=\{\bu\in U:\exists \by\in\real^m_+ ,\; \bba\bx+\bbb\by+\bbc\bv\geq \bc \}$, and 
maximizing the LHS of \eqref{eq:CS_breakdown} over $\bu$ we obtain that
$$\ba^T\bx+\max_{\bu\in U}\max_{\substack{\by\in Y(\bx^k,\bu),\bw\in\real^r_+,\\
\bbb^T\bw\leq \bb,(\by,\bw)\in C(\bx,\bu)}}\bb^T\by=\ba^T\bx+\max_{\substack{\bu\in U\bigcap U(\bx^k),\\
\by\in Y(\bx^k,\bu),\bw\in\real^r_+,\\
\bbb^T\bw\leq \bb,(\by,\bw)\in C(\bx^k,\bu)}} \bb^T\by= \underline{Z}(\bx^k,\{\bu^k\})<\infty.$$
However, if $\bx^k$ is not feasible, we have that $$\tilde{Z}(\bx^k)=\underline{Z}(\bx^k,\{\bu^k\})<\underline{Z}(\bx^k,\overline{\bu}(\bx^k))=\overline{Z}(\bx^k)=\infty.$$

If $U$ is polyhedral, as given in \eqref{eq:Define_U}, the maximizer $\bu^k=\tilde{\bu}(\bx^k)$ (which must exist since $U$ is compact)
corresponds to one of the vertices of the lifted polytope
$$\{(\bu,\by)\in\real^l\times\real^m:\bba\bx^k+\bbb\by+\bbc\bv\geq \bc,\bbd\bu\leq \bd\}$$
which, in general, is not necessarily a vertex of $U$. 
\end{proof}
Thus, the importance of a feasibility oracle in the CCG framework for problems satisfying Assumption~\ref{ass:feasibility} rather than Assumption~\ref{ass:full recourse} is apparent.
In the next subsection we suggest such an oracle based on complementary slackness and discuss its use.
\subsection{Complementarity Based Feasibility Oracle}
In order to construct a feasibility oracle we will first notice that for a given $\bx$ and $\bu$, determining if $Y(\bx,\bu)=\emptyset$ is equivalent to checking if the optimization problem \eqref{eq:SecondStage_Feas} has a strictly positive  objective function value:
\begin{equation}\label{eq:SecondStage_Feas}
\begin{aligned}
\underline{\alpha}(\bx,\bu)=&\min_{(\alpha,\by)\in\real_+\times\real^m_+} \alpha\\
&\text{s.t.} &\bba\bx+\bbb\by+\bbc\bu+\alpha \bone\geq \bc. 
\end{aligned}
\end{equation} 
Moreover, notice that for any $\bx$ and $\bu\in U$ Problem~\eqref{eq:SecondStage_CSequiv} is feasible. Denoting $\tilde{\by}=(\by,\alpha)$  $\tilde{\bb}=\bone_{m+1}$, $\tilde{\ba}=\bzero$, and $\tilde{\bbb}=[\bbb,\bone]$, Problem~\eqref{eq:SecondStage_CSequiv} can be viewed as the second-stage of a two-stage RO problem of type \eqref{eq:2stage_robust} (where $\by,\ba,\bb,$ and $\bbb$ are replaced by $\tilde{\by}$, $\tilde{\ba}$, $\tilde{\bb}$ and $\tilde{\bbb}$, respectively), which satisfies assumption~\ref{ass:full recourse}.
Thus, applying the optimality oracle suggested in Proposition~\ref{prop:CCG_equiv} to this problem will produce a value 
$\overline{\alpha}(\bx)=\max_{\bu\in U} \underline{\alpha}(\bx,\bu)$ and a maximizer ${\bu}^\alpha(\bx)=\argmax_{\bu\in U} \underline{\alpha}(\bx,\bu)$ such that:
\begin{itemize}
\item $\bx$ is feasible if and only if $\overline{\alpha}(\bx)=0$.
\item ${\bu}^{\alpha}(\bx)$ can always be chosen as an extreme point of $U$ (as a result of Lemma~\ref{lemma:Maximizer_ExtremePoint}).
\end{itemize}
Thus, we can rewrite the definition of $\overline{Z}(\bx)$ and $\overline{\bu}(\bx)$ as follows.
\begin{equation}\label{eq:RealWorstCase}\overline{Z}(\bx)=\begin{cases}
\infty, & \text{if } \overline{\alpha}(\bx)>0,\\
\tilde{Z}(\bx), & \text{otherwise},
\end{cases}\end{equation} 
and 
\begin{equation}\label{eq:RealWorstCase_Realization}\overline{\bu}(\bx)=\begin{cases}
{\bu}^\alpha(\bx), & \text{if } \overline{\alpha}(\bx)>0,\\
\tilde{\bu}(\bx), & \text{otherwise}.
\end{cases}\end{equation} 
Thus, using similar arguments to those in the proof of Proposition~\ref{prop:CCG_converge} we have the following result.
\begin{corollary}
Let \eqref{eq:RealWorstCase} and \eqref{eq:RealWorstCase_Realization} be used to compute $\overline{Z}(\cdot)$ and $\overline{\bu}(\cdot)$, respectively, in Algorithm~\ref{alg:CCG}, and let Assumption~\ref{ass:feasibility} hold. 
If $U$ is a polytope, then $\overline{\bu}(\bx)$ can always be chosen as a vertex of $U$, and Algorithm~\ref{alg:CCG} terminates in a finite number of steps with an optimal solution for Problem \eqref{eq:2stage_robust}.
\end{corollary}

Although we can use Proposition~\ref{prop:CCG_equiv} to construct an MIO problem, similar to \eqref{eq:SecondStage_WorstCaseMIO}, in order to compute $\overline{\alpha}(\bx)$, it is practically much harder to solve. Identifying the tolerance $\epsilon>0$ for which we can determine that $\overline{\alpha}(\bx)=\epsilon>0$ is nontrivial. Moreover, in numerical experiments we found that while computing $\tilde{Z}(\bx)$ takes seconds, the optimization problem used to compute $\overline{\alpha}(\bx)$, for the same problem instance, does not solve even after 20-30 minutes. Therefore, in the next section, we take a different approach to modifying Algorithm~\ref{alg:CCG}, in order to ensure both feasibility of the resulting first-stage decision and convergence, under Assumption~\ref{ass:feasibility}.

\section{Duality Driven Bender Decomposition (DDBD)}\label{sec:DDBD}
In this section, we will describe a modification of the CCG algorithm which guarantees convergence to an optimal solution of Problem~\eqref{eq:2stage_robust} under Assumption~\ref{ass:feasibility}. Since, in general, finding whether $\bx$ is feasible for Problem~\eqref{eq:2stage_robust} is a hard problem, our method utilizes two types algorithms:
\begin{enumerate}
\item A fast algorithm $\mathcal{F}_1$ that, given $\bx\in X$, returns a point $\bu^{\mathcal{F}}(\bx)$ 
such that $\tilde{Z}(\bx)\leq \underline{Z}(\bx,\{\bu^{\mathcal{F}}(\bx)\})\leq \overline{Z}(\bx)$. It follows from Proposition \ref{prop:CCG_equiv} that if $\bx$ is feasible then $\underline{Z}(\bx,\{\bu^{\mathcal{F}}(\bx)\})= \overline{Z}(\bx)$. We also assume that if $U$ is a polytope, the output of $\mathcal{F}_1(\bx)$ must be a vertex of $U$. We will refer to $\mathcal{F}_1$ as the \emph{fast feasibility oracle}.
\item A slow algorithm $\mathcal{F}_2$ that verifies the feasibility of $\bx$, and in the case $\bx$ is infeasible, returns an uncertainty realization $\overline{\bu}(\bx)$ such that $\overline{Z}(\bx)=\underline{Z}(\bx,\{\overline{\bu}(\bx)\})$. Notice that by the definition of $\overline{\bu}(\bx)$ given in \eqref{eq:RealWorstCase_Realization} and Lemma~\ref{lemma:Maximizer_ExtremePoint} we can assume w.l.o.g. that
$\overline{\bu}(\bx)$ is a vertex of $U$. We will refer to $\mathcal{F}_2$ as the \emph{exact feasibility oracle}.
\end{enumerate}

Given these oracles we propose the following algorithm. 
\begin{algorithm}\caption{Duality Driven Bender Decomposition (DDBD)}\label{alg:DDBD}
		\begin{itemize}
				\item[$\rm{(1)}$] {\bf Input:} $X,\bba,\bbb,\bbc,\ba,\bb,\bc,\bbd,\bd,\bu^0\in U.$
			    \item[$\rm{(2)}$] {\bf Initialize:} ${UB}^0=\infty$, ${LB}^0=-\infty$, $V^0=\{\bu^0\}, k=0$
			    \item[$\rm{(3)}$] While ${LB}^k< {UB}^k$
			    \begin{itemize}
			        \item[$\rm{(a)}$] Update $k\leftarrow k+1$.
			        \item[$\rm{(b)}$] Compute $\bx^k=\argmin_{\bx\in X} \underline{Z}(\bx,V^{k-1})$ and its corresponding lower bound value ${LB}^k=\underline{Z}(\bx^k,V^{k-1})$.
			        \item[$\rm{(c)}$] Compute $\bu^{k}=\mathcal{F}_1(\bx^k)$ and the corresponding upper bound ${{UB}^k=\underline{Z}(\bx^k,\{\bu^k\})}$.
			        \item[$\rm{(d)}$] Update $V^{k}\leftarrow V^{k-1}\bigcup\{\bu^{k}\}$.           
			     \end{itemize}
			     \item[$\rm{(4)}$] Compute $\overline{\bu}(\bx^k)=\mathcal{F}_2(\bx^k)$ and $\overline{Z}(\bx^k)=\underline{Z}(\bx,\{\overline{\bu}(\bx^k)\})$.\\
			     If $\overline{Z}(\bx^k)=\infty$
			     set $V\leftarrow V\bigcup\{\overline{\bu}(\bx^k)\}$ and $UB=\infty$ and go to Step 3.  
			     Otherwise return $\bx^k$ and $\overline{Z}(\bx^k)$.      		
        \end{itemize} 
\end{algorithm}

\begin{theorem}
Let $\mathcal{F}_1$ and $\mathcal{F}_2$ be the fast and exact feasibility oracles in Algorithm~\ref{alg:DDBD}, respectively, and let Assumption~\ref{ass:feasibility} hold. 
If $U$ is a polytope, then Algorithm~\ref{alg:DDBD} terminates in a finite number of steps and return an optimal solution of Problem~\eqref{eq:2stage_robust}. 
\end{theorem}
The proof is similar to that of Proposition~\ref{prop:CCG_converge} and will be omitted.
We will now present specific algorithms for $\mathcal{F}_1$ and $\mathcal{F}_2$, which satisfy the restrictions given above for the case $U$ is a polytope.

\subsection{Algorithm for $\mathcal{F}_1$}\label{sec:F1}
Let $\bx\in X$ be some first-stage decision of Problem~\eqref{eq:2stage_robust}. We saw that
$$\tilde{Z}(\bx)=\underline{Z}(\bx,\{\tilde{\bu}(\bx)\})\leq \overline{Z}(\bx),$$ 
where the inequality is satisfied with equality if and only if $\bx$ is feasible ($\overline{Z}(\bx)<\infty$).

If Assumption~\ref{ass:feasibility} holds and $U$ is compact, then for any $\bx\in X$ and $\bu\in U$ the second-stage problem is feasible if and only if $Y(\bx,\bu)\neq \emptyset$, which is true if and only if the following Problem~\eqref{eq:dual_inner} is unbounded \footnote{Notice that under Assumption~\ref{ass:feasibility} Problem~\eqref{eq:dual_inner} can not be infeasible.}. 
\begin{equation}\label{eq:dual_inner}
\max_{\bw:\bbb^T\bw\leq \bb, \bw\geq 0} (\bc-\bba\bx-\bbc \bu)^T\bw.
\end{equation}
This is equivalent to finding a vector in the recession cone of the feasible set for which the objective is strictly positive, \ie checking if the following optimization Problem~\eqref{eq:unbounded_dual_inner} has a strictly positive optimal objective function value.
\begin{equation}\label{eq:unbounded_dual_inner}
\max_{\bw:\bbb^T\bw\leq \bzero, \bzero\leq\bw\leq \bone} (\bc-\bba\bx-\bbc \bu)^T\bw.
\end{equation}
Therefore, if $\bx$ is infeasible, maximizing Problem~\eqref{eq:unbounded_dual_inner} over $\bu$ will result in the bilinear optimization Problem~\eqref{eq:dual_inner_max_u}, which has a strictly positive optimal objective function value:
\begin{equation}\label{eq:dual_inner_max_u}
\max_{\bu\in U,\bw:\bbb^T\bw\leq \bzero, \bzero\leq\bw\leq \bone} (\bc-\bba\bx-\bbc \bu)^T\bw.
\end{equation}
Since Problem~\eqref{eq:dual_inner_max_u} is not convex, we must apply some heuristic to solve it. Moreover, since the problem is bilinear, an easy solution is applying the alternating maximization (AM) method as described in Algorithm \ref{alg:Alternating_Max}. 

\begin{algorithm}[H]\caption{Alternating Maximization (AM)}\label{alg:Alternating_Max}
		\begin{itemize}
       		\item[$\rm{(1)}$] {\bf Input:} $\bba,\bbb,\bbc,\bc,\bbd,\bd, \tilde{\bu}(\bx)$
       		\item[$\rm{(3)}$]{\bf Initialize:} $f^{0}=-\infty$, $\bw^*=\argmax_{\bw:\bbb^T\bw\leq \bb, \bw\geq \bzero}(\bc-\bba\bx-\bbc\tilde{\bu}(\bx))^T\bw$, 
            \item[$\rm{(3)}$] While $(f^k-f^{k-1})> 0$ or $k=0$
            \begin{itemize}
            \item[$\rm{(a)}$] Update $k\leftarrow k+1$.
            \item[$\rm{(b)}$] Find $\bu^*\in\argmax_{\bu\in U} (\bc-\bba\bx-\bbc\bu)^T\bw^*$.
            \item[$\rm{(c)}$] Find $\bw^*\in\argmax_{\bw:\bbb^T\bw\leq \bzero, \bzero\leq \bw\leq \bone} (\bc-\bba\bx-\bbc\bu^*)^T\bw$. 
            \item[$\rm{(d)}$] Compute $f^k=(\bc-\bba\bx-\bbc\bu^*)^T\bw^*$.            
            \end{itemize}     
        \item[$\rm{(4)}$] Return  $\bu^{\mathcal{F}}(\bx)=\bu^*$.     
        \end{itemize} 
\end{algorithm}

Starting Algorithm~\ref{alg:Alternating_Max} from point $\tilde{\bu}(\bx)$, which generates the maximal objective function value for Problem~\eqref{eq:SecondStage_WorstCaseMIO}, we alternate between solving the problem for $\bw$ given $\bu$ and solving the problem for $\bu$ given $\bw$. We stop when we cannot improve the objective function further.
We know, by the fact that the AM method is monotone and that the objective function is continuous (specifically bilinear), that if $U$ is convex and compact, then the value will converge to some limit in sublinear time (see for example \citep[Theorem 10]{shtern2016computational}), and that all
limit points of the algorithm are stationary points of the problem \citep[Corollary 2]{Grippo2000}. However, there is no guarantee that the AM method will converge to the true optimal value. Therefore, if this algorithm terminates with a nonpositive value it does not mean that $\bx$ is feasible, but rather that we can not prove infeasibility, and so this algorithm, which will take the role of $\mathcal{F}_1(\bx)$, is a fast but inexact way of determining the feasibility of $\bx$. 

Notice that Algorithm~\ref{alg:Alternating_Max} is well defined for any compact $U$. Since $\tilde{\bu}(\bx)$ is feasible for Problem~\eqref{eq:SecondStage_WorstCaseMIO}, the maximizer $\bw^*$ exists, and $\bu^*$ exists due to the compactness of $U$. Next we prove that Algorithm~\ref{alg:Alternating_Max} satisfies the assumptions we made about $\mathcal{F}_1$.

\begin{proposition}
Let $\bu^{\mathcal{F}}(\bx)$ be the output of Algorithm~\ref{alg:Alternating_Max}, then
$$\underline{Z}(\bx,\{\bu^{\mathcal{F}}(\bx)\})\geq \underline{Z}(\bx,\{\tilde{\bu}(\bx)\})\equiv\tilde{Z}(\bx).$$
If $U$ is a polytope then $\bu^{\mathcal{F}}(\bx)$ can always be chosen as a vertex of $U$.
\end{proposition}
\begin{proof}{Proof.}
The first claim is trivial, since if $f^k>0$ then by construction $\underline{Z}(\bx,\{\bu^*\}))=\infty$, otherwise, if $f^k=0$, by the initialization of
$\bw^*$ we have that $\underline{Z}(\bx,\{\bu^*\}))\geq \underline{Z}(\bx,\{\tilde{\bu}(\bx)\}$.
If $U$ is a polytope then $\bu^*$ is a solution of a linear optimization problem, and thus can always be chosen to be a vertex of $U$. 
\end{proof}
 
\subsection{Algorithm for $\mathcal{F}_2$}\label{sec:F2}
In order to prove or disprove infeasibility in the case where $U$ is a polytope defined by~\eqref{eq:Define_U}, we suggest looking at the dual formulation of Problem \eqref{eq:2stage_robust} in $(\bu,\by)$ suggested in \citep{Bertsimas2015}. 
\begin{equation}\label{Problem:Two Stage Dual}
\begin{aligned}[t]
&\min_{\bx\in X} &\ba^T\bx+\max_{\bw\in W}\min_{\blambda\geq 0:\bbc^T\bw+\bbd^T\blambda\geq 0} (\bc-\bba\bx)^T\bw &+\bd^T\blambda,\\
\end{aligned}
\end{equation}
where $W=\paren{\bw\in\real^r: \bw\geq0,\bbb^T\bw\leq \bb}$.
Notice that in this problem, $\bx$ does not impact the feasibility of a certain solution, only the value of the objective function. Moreover, for each fixed $\bw\in W$ the optimal value of $\blambda$ is simply given as a solution of the following linear optimization problem (independent of $x$)
\begin{equation}\label{Problem:Lambda given w}
\begin{aligned}[t]
&\min_{\blambda\in \real^d} &\bd^T\blambda&\\
&\quad\text{s.t.}&\bbc^T\bw+\bbd^T\blambda&\geq \bzero,\\
&&\blambda&\geq \bzero.\\
\end{aligned}
\end{equation}
Therefore, we can define an extended $\tilde{\blambda}$ variable which includes both the original $\blambda$ as well as the slack variables, and allows us to transform the problem to the following standard form.
\begin{equation}\label{Problem:Problem in lambda}
\begin{aligned}[t]
&\min_{\blambda\in \real^d} &\tilde{\bd}^T\tilde{\blambda}&\\
&\quad\text{s.t.}&\bbc^T\bw+\tilde{\bbd}^T\tilde{\blambda}&= \bzero,\\
&&\tilde{\blambda}&\geq \bzero.\\
\end{aligned}
\end{equation}
where $\tilde{\bd}=[\bd;\bzero]$ and $\tilde{\bbd}=[\bbd, -\bbi]$. Assuming that $U$ is nonempty, Problem \eqref{Problem:Problem in lambda} must also be feasible, and so there exists a set of independent columns, the index set of which we denote by $I$, such that
$\tilde{\blambda}_I^*=-\tilde{\bbd}_I^{-1}\bbc^T\bw$, and this basis is optimal for all $\bw$ which belong to 
$$W_I=\paren{\bw\in W:-\tilde{\bbd}_I^{-1}\bbc^T\bw\geq \bzero}.$$ Thus, it follows that the optimal $\blambda(\bw)$ is a piecewise linear function of $\bw$, and independent of $\bx$. Identifying all the optimal bases and regions $W_I$, we can subsequently find the optimal two-stage strategy for this dual problem. However, since the number of these bases might be exponential, identifying only the bases which generate the worst case is important.

In order to utilize these facts and construct an algorithm we use the partitioning concept suggested in \citep{BertsimasDunning2014,PostekDenHertog2014}. Assuming a given partition of $W$ to $\paren{W^p}_{p\in\mathcal{L}}$ such that $$W^p=\paren{\bw\in\real^r: \bw\geq \bzero,\; \bbb^p\bw\leq\bb^p},\; W=\bigcup\limits_{p\in\mathcal{L}} W^p,$$ we assume a second-stage policy $\lambda^p(\bw)$ which is linear in $\bw$, \ie
$$\lambda^p(\bw)=\bbz^p\bw+\bz^p.$$
Thus, applying this linear rule to each partition element $p\in\mathcal{L}$ results in Problem  \eqref{Prob:LinearDecision_Partition_Dual}, and taking $\tau^*=\max\limits_{p\in\mathcal{L}} \tau^p$, we obtain an upper bound on Problem \eqref{Problem:Two Stage Dual}, where $\tau^p$ is defined as
\begin{equation}\label{Prob:LinearDecision_Partition_Dual}
\begin{aligned}[t]
&\min_{\tau^p,\bbz^p} &\tau^p&\\
&\text{s.t.}\quad &(\bc-\bba\bx)^T\bw +\bd^T\bbz^p\bw+\bd^T\bz^p+\ba^T\bx&\leq \tau^p,\quad\forall \bw\in W^p\\\
&&\bbc^T\bw+\bbd^T\bbz^p\bw+\bbd^T\bz^p&\geq \bzero,\quad\forall \bw\in W^p\\
&&\bbz^p\bw+\bz^p&\geq \bzero,\quad\forall \bw\in W^p.\\
\end{aligned}
\end{equation}
Problem \eqref{Prob:LinearDecision_Partition_Dual} can then be reformulated as a linear optimization problem using the robust counterpart mechanism described in \citep{BenTal2009}.
Notice that for a given $\bx$, $\tau^*$ admits an upper bound on the value of Problem~\eqref{eq:2stage_robust}, regardless of the partition $\paren{W^p}_{p\in\mathcal{L}}$. Thus, if at any stage of the partitioning scheme \eqref{Prob:LinearDecision_Partition_Dual} admits a finite value for all $p\in\mathcal{L}$ it follows that the given $\bx$ is feasible. Verifying infeasibility, however, is a more complex task. 

We suggest a partitioning scheme, different from the ones presented by \citet{BertsimasDunning2014} and \citet{PostekDenHertog2014}, which is based on identifying optimal bases. The full partitioning algorithm is given in Algorithm~\ref{alg:Dcuts}. In each stage of the algorithm we identify the active partition elements set $\mathcal{A}$ defined as
$$\mathcal{A}=\paren{p\in \mathcal{L}: \tau^p=\tau^*}.$$ In the case where $\bx$ is not verified to be feasible, the upper bound  $\tau^*$ will be infinity (unbounded problem). For each $p\in \mathcal{A}$ we construct its sub-partition by first finding a point $\bw\in\inter{(W^p)}$ and its corresponding optimal basis $I^p$, and define $\tilde{\bbb}=-\tilde{\bbd}_{I^p}^{-1}\bbc^T$. We then define its primary sub-partition with the aid of matrix $\tilde{\bbb}$. Without loss of generality, we assume that matrix $\tilde{\bbb}\in\real^{k\times l}$ has $k\leq r$ rows which are nonzeros (since all zero rows can be eliminated), and we partition $W^p$ into $k+1$ partition elements, a primary partition element with partitioning index $\ell=0$, given by
\begin{equation}\label{eq:primary_part} \tilde{W}^{0}=W\cap \paren{\bw: \tilde{\bbb}\bw\geq \bzero},\end{equation}
and $k$ secondary partition elements, such that for $j=1,\ldots, k$ the partition with partitioning index $\ell=j$ is defined by 
\begin{equation}\label{eq:secondary_part} \tilde{W}^j=W^p \cap \paren{\bw: \tilde{\bbb}_i\bw\geq 0, i=1,\ldots,j-1,\;  \tilde{\bbb}_j\bw\leq 0 }.\end{equation}

The following proposition states that if at any time in the algorithm a partition element that cannot be  partitioned further has an unbounded objective value, then $\bx$ is infeasible for Problem~\eqref{eq:2stage_robust}.
\begin{proposition}
If at some stage of Algorithm~\ref{alg:Dcuts}, for which $\tau^*=\infty$, there exists a partition $p\in\mathcal{A}$ with partition index $\ell^p=0$, then $\bx$ is infeasible for Problem~\eqref{eq:2stage_robust}.
\end{proposition}
\begin{proof}{Proof.}
Since for any $p\in \{\tilde{p}\in \mathcal{A}:\ell^{\tilde{p}}=0\}$ the optimal solution for any $\bw\in W^p$ is given by the basis $I^p$, \ie $\blambda(\bw)=-\tilde{\bbd}_{I^p}^{-1}\bbc^T\bw$. Since the optimal value on partition $p$ is unbounded, it follows that the optimal value of Problem~\eqref{Problem:Problem in lambda} is unbounded. This implies that $\bx$ is infeasible for Problem \eqref{Problem:Two Stage Dual} and hence to Problem~\eqref{eq:2stage_robust}. 
\end{proof}
\begin{algorithm}[h]\caption{Dual Basis Cuts (DBC)}\label{alg:Dcuts}
		\begin{itemize}
       		\item[$\rm{(1)}$] {\bf Input:} $\bba,\bbb,\bbc,\ba,\bb,\bc,\bbd,\bd, \epsilon>0,\bx,\tilde{\bu}(\bx).$
            \item[$\rm{(2)}$] {\bf Initialize:} $P=1$ $\bbb^1=\bbb^T$, $\bb^1=\bb$, $l^1=-1$, $\text{Children}^1=\emptyset$,$it=1$.  
            \item[$\rm{(3)}$] For $it=1,2,\ldots$ 
            \begin{itemize}           
            \item[$\rm{(a)}$] Let $\mathcal{L}=\paren{p\in\{1,\ldots,P\}:\text{Children}^p=\emptyset}$ be the set of leafs.
            \item[$\rm{(b)}$] For each $p\in\mathcal{L}$ Solve \eqref{Prob:LinearDecision_Partition_Dual} and obtain $(\tau^p,\bbz^p,\bz^p)$. 
            \item[$\rm{(c)}$] Calculate $\tau^*\leftarrow\max_{p\in\mathcal{L}}\tau^p$ 
            \item[$\rm{(d)}$] If $\tau^*<\infty$ {\bf exit} and return $\overline{\bu}(\bx)=\tilde{\bu}(\bx)$.
            \item[$\rm{(e)}$] Else for any $p\in\mathcal{A}$ (\ie $\tau^p=\tau^*$)
            \begin{itemize}
             \item[$\rm{(i)}$] If $l^p=0$ then, find $\tilde{\bw}$ an unbounded ray in the partition (see \eqref{eq:FindingUnboundedRay_Partition})  and return $\overline{\bu}(\bx)\in\argmax_{\bu\in U} (\bc-\bba\bx-\bbc\bu)^T\tilde{\bw}$.
             \item[$\rm{(ii)}$] Else $l^p>0$ then            
             \begin{itemize}
              \item Find $\bw^p\in \inter(W^p)$             
              \item Find $\blambda^p=\arg\min_{\blambda\geq 0:\bbc^T\bw^p+\bbd^T\blambda\geq0} \bd^T\blambda$
              \item Identify basis $I^p$ associated with solution $\blambda^p$ and corresponding $\tilde{\bbd}_{I^p}$. 
              \item Set $\tilde{\bbb}=\tilde{\bbd}_{I^p}^{-1}\bbc^T$ (while eliminating rows $j$ for which $\tilde{\bbb}_j=\bzero$). Let $k$ be the number of rows of $\tilde{\bbb}$ ($k\leq d$ where $d$ are the number of rows in $\bbd$).   
              \item For $j=0,\ldots,k$:\\ 
              If $j=0$ then $P(j)=1$ otherwise $P(j)=p$\\
              If  $\inter(\tilde{W}^j)\neq \emptyset$ (defined in \eqref{eq:primary_part}-\eqref{eq:secondary_part}) then $P\leftarrow P+1$, $\ell^P\leftarrow j$, $W^{P}\leftarrow\tilde{W}^j$, 
               and update $\text{Children}^P\leftarrow\emptyset$ and $\text{Children}^{P(j)}\leftarrow\text{Children}^p\bigcup\{P\}$. 
              \end{itemize}        
            \end{itemize}
             \end{itemize}
        \end{itemize} 
\end{algorithm}

Now that we know how to identify infeasibility, we need to find a way to extract an unbounded ray for the case where a partition with $\ell^p=0$ is unbounded. For that purpose we define the optimal $\bbz^p=-\bbi^p\tilde{\bbd}_{I^p}^{-1}\bbc^T$ and $\bz^p=\bzero$, where $\bbi^p$ is a transformation from the  basis $I^p$ to the original variables. Since $\lambda(\bw)=\bbz^p\bw$ is an optimal solution for this partition we are only left to find a ray for which this solution is unbounded. To do this, we solve the following optimization problem
\begin{equation}\label{eq:FindingUnboundedRay_Partition}
\tilde{\bw}\in\argmax_{\bw:\bbb^p\bw\leq 0,0\leq\bw\leq 1} (\bc-\bba\bx+(\bbz^p)^T\bd)^T\bw,
\end{equation}
which is guaranteed to be positive (since the problem is unbounded on this partition).
As the algorithm suggests, we can now take $\tilde{\bw}$ and find 
$$\overline{\bu}(\bx)\in\argmax_{\bu\in U} (\bc-\bba\bx-\bbc\bu)^T\tilde{\bw}.$$ 
Notice that if $U$ is polyhedral we can always choose $\overline{\bu}(\bx)$ to be a vertex of $U$.
\begin{remark}\label{rem:Dcuts}\leavevmode
\begin{enumerate}[label=({\roman*})]
\item The calculation of the upper bound done in Step 3b of Algorithm~\ref{alg:Dcuts} can be done in parallel for each partition element, and thus both the computational and storage overhead required by Algorithm~\ref{alg:Dcuts} can be reduced.
\item For partition elements $p\in \mathcal{L}$ with partition index $\ell^p=0$ we have that $\bbz^p=-\bbi^p\tilde{\bbd}_{I^p}^{-1}\bbc^T$ and $\bz^p=\bzero$ independent of $\bx$, and thus these values can be computed only once.
\item Algorithm~\ref{alg:Dcuts} with slight modifications can also be used to find the optimal $\bx$ and not only to verify if $\bx$ is optimal. Much like the partitioning algorithms presented in \citep{BertsimasDunning2014,PostekDenHertog2014}, a feasible $\bx$ and an upper bound on its worst case value can be found by solving the following semi-infinite problem
\begin{equation}\label{Prob:LinearDecision_Partition_Dual_x}
\begin{aligned}[t]
&\min\limits_{\substack{\bx,\tau^*,\\\{(\tau^p,\bbz^p,\bz^p)\}_{p\in\mathcal{L}}}} \tau^* \\
&\begin{array}{lrll}
\qquad\text{s.t.}\quad &\tau^*&\geq \tau^p,&\forall p\in\mathcal{L}\\
&(\bc-\bba\bx)^T\bw +\bd^T\bbz^p\bw+\bd^T\bz^p+\ba^T\bx&\leq \tau^p,&\forall \bw\in W^p,\;\forall p\in\mathcal{L},\\
&\bbc^T\bw+\bbd^T\bbz^p\bw+\bbd^T\bz^p&\geq0,&\forall \bw\in W^p,\;\forall p\in\mathcal{L},\\
&\bbz^p\bw+\bz^p&\geq 0,&\forall \bw\in W^p,\;\forall p\in\mathcal{L},
\end{array}
\end{aligned}
\end{equation}
which can be reformulated as an MIO problem.
We can then replace Step 3b by solving Problem~\eqref{Prob:LinearDecision_Partition_Dual_x}, and obtain a lower bound  by solving Problem~\eqref{Prob:LinearDecision_Partition_Dual_x} with $\mathcal{L}$ replaced by 
$\mathcal{L}^0=\{p\in\mathcal{L}:\ell^p=0\}$, \ie the partition elements with partition index $0$, since for these elements the optimal second-stage strategy has already been found. The termination condition in Step 3d will be replaced by $(UB-LB)<\epsilon$ for some tolerance $\epsilon$. Moreover, Step 3e(i) is dropped and in Step 3e(ii) $\bw$ is chosen to be $$\bw^p\in\argmax_{\bw\in W^p}(\bc-\bba\bx)^T\bw +\bd^T\bbz^p\bw+\bd^T\bz^p+\ba^T\bx,$$ provided that the resulting primary partition element has a nonempty interior, and some $\bw^p\in\inter(W^p)$ otherwise. We refer to this method as the \emph{Dual Basis Cuts} (DBC) method which we will also compare with in the numerical section. 
\end{enumerate}
\end{remark}

The next theorem addresses the finite termination of Algorithm~\ref{alg:Dcuts}.
\begin{theorem}
Algorithm~\ref{alg:Dcuts} will terminate after a finite number of iterations.
\end{theorem}
\begin{proof}{Proof.}
The algorithm produces a tree where at each iteration one node of the tree is explored creating at most $d$ children. Thus, the maximal number of iterations $K$ is bounded by the number of possible tree nodes. Denoting $h$ as the maximal depth of the tree, the number of nodes is bounded by $\sum_{i=0}^{h} d^{i}=\frac{d^h-1}{d-1}$. 
Since the children of each node are derived from a basis, which can not be equal to any of the bases that created their ancestor nodes, it follows that the depth of the tree $h$ is bounded above by the number of feasible bases. The number of feasible bases is bounded above by ${{d+l}\choose{l}}\leq {2^{d+l}}$. Thus, we have that
$h\leq {2^{d+l}}$ and $K\leq \frac{d^{2^{d+l}}-1}{d-1}$ is finite. 
\end{proof}

\section{Numerical Results}\label{sec:NumericalResults}
In this section, we will examine the performance of DDBD (Algorithm~\ref{alg:DDBD}) for several examples. We compare our results to the AMIO partitioning scheme presented by \citet{BertsimasDunning2014}, as well as to the partition scheme based on the DBC (Algorithm~\ref{alg:Dcuts}) we suggested in the previous section combined with Remark~\ref{rem:Dcuts}(iii). 
Since all the algorithms suggest a lower bound (LB) and upper bound (UB), we looked at the UB to LB relative gap which is given by 
$$\text{UL gap}=\frac{UB-LB}{\max\{\min\{|LB|,|UB|\},1\}}$$

In the implementation of the AMIO methods we generated the added constraints for each part of the tree rather than generating it recursively at each node. We did so because, in this case, time constrictions proved to be more crucial than storage limitations. All methods were implemented in Julia v0.4.2 using Gurobi v6.5.1 solver. The algorithms were run on one core of a 32 processor AMD64 with 246.4G RAM.

We compare the methods' performance in three examples: the \emph{location-transportation problem}, with structure presented by \citet{AtamturkZhang2007}, in which the first-stage variable $\bx$ is partially binary; the \emph{capacitated network lot-sizing problem}, for which both structure and data is presented by \citet{Bertsimas2015} and all variables are continuous; and the \emph{unit commitment problem} presented by \citet{Bertsimas2013} with partially binary first-stage decision variables. 

\subsection{The Location-Transportation Problem}
Consider the problem in which $N$ possible facilities supply the demand of $L$ customers. For each facility $i$ there is fixed cost $a^{f}_i$ for opening the facility, a variable cost  $a^v_i$ for each unit produced in the facility, and a maximal production capacity $\sigma_i$. Furthermore, the transportation cost between facility $i$ and customer $j$ is given by $b_{ij}$.
Moreover, each customer $j$ has an unknown demand which can be represented as
$d_j=d^{min}_j+\delta_j u_j$ for some $0\leq u_j\leq 1$, which satisfy $\sum_j u_j \leq \gamma L$ where $\gamma\in (0,1)$, so the uncertainty set for the demand is given by
$$\mathcal{D}=\paren{\bd:\exists \bu\in\real^L_+,\bd=\bd^{min}+\bdelta \bu, \bu\leq \bone, \bone'\bu\leq \gamma L}.$$
The decision regarding which of the facilities to open and how much to produce in those facilities is made prior to the realization of the demand. Once the demand is realized the allocation of the units to the various customers is done so to minimize the transportation cost. The formulation of this problem is given as follows.
{\begin{align}\label{Prob:Location Transportation}
&\min_{\bx\in\real^N_+,\bz\in \integ^N,\by(\bd)}\quad \sum_{i=1}^{N} \left(a^{v}_ix_i+a^{f}_iz_i\right)+\left(\max_{\bd\in\mathcal{D}}\sum_{i=1}^{N}\sum_{j=1}^L b_{ij}y_{ij}(\bd)\right)
\end{align}
\begin{equation*}
\begin{aligned}
&\text{s.t.}& \sum_{i=1}^N y_{ij}(\bd)&\geq d_j, && j=1,\ldots,L,\;\forall \bd\in \mathcal{D},\\
&& x_i-\sum_{j=1}^L y_{ij}(\bd)&\geq 0, && i=1,\ldots,N,\;\forall \bd\in \mathcal{D},\\
&&\by(\bd)&\geq 0,&&\forall \bd\in \mathcal{D},\\
&& x_i&\leq \sigma_iz_i,&& i=1,\ldots,N.
\end{aligned}
\end{equation*}
}

We generated $100$ realization for this problem with $N=L=10$, where for each realization and each $i,j\in\{1,\ldots,10\}$, the problem's parameters are chosen uniformly at random from the following ranges $a^{f}_i\in[1,10]$,$a^v_i\in[0.1,1]$,
$b_{ij}\in[0,10]$, $\sigma_i\in [200, 700]$,
$d^{min}_j\in[10,500]$, $\delta_j\in[0.1,0.5]\cdot d^{min}_j$, and $\gamma=0.5$. To ensure feasibility, we chose the parameters so that the condition $\sum_i \sigma_i \geq \sum_j d_j$ is satisfied for any possible realization in the uncertainty set.

We found that, for this problem and under this choice of parameters (as well as many others), the affine decision rule is often optimal or close to optimal. This therefore raises the question of how fast does the algorithm detect that this is the case. All the algorithms were run until either reaching a gap of $10^{-3}$ or $1000$ seconds, the earlier of the two.

We first compared the performance of the CCG algorithm, which does not ensure feasibility, to the DDBD algorithm. The results are given in Table~\ref{tbl:CCG_Compare_LocationTransportation_N=L=10}. We see that for $31\%$ of the instances tested the CCG algorithm converged to an infeasible solution, while the DDBD always converged to a feasible solution without the need to utilize Step 4 of Algorithm~\ref{alg:DDBD}. Moreover, checking the DDBD solution is indeed feasible requires less than one second for all the instances. The average convergence time for the instances which converged to a feasible solution (equivalently optimal solution) was longer for the CCG than for the DDBD.
\begin{table}[!h]
\centering
\caption{Feasibility and runtime summary for the location-transportation problem with $N=L=10$.}\label{tbl:CCG_Compare_LocationTransportation_N=L=10}
{\renewcommand{\arraystretch}{0.7}
\setlength{\extrarowheight}{.1em}
\begin{tabular}{|c||c|c|}
\hline
Algorithm & \% Feasible & Mean (std) time (sec)\\
\hline
CCG & 69\% & 27.22 (15.78)\\
DDBD & 100\% & 22.59 (17.81)\\
\hline
\end{tabular}}
\end{table}

Next, we compare the performance of the DDBD to the AMIO and DBC algorithms. We also ran the AMIO algorithm on the dual problem, but this resulted in worse performance than the AMIO applied on the primal problem for all instances. Figure~\ref{Fig:TerminateStats__LocationTransportation_N=L=10} shows the number of instances not terminated (for any reason) by a specific time (left figure) and the termination reason statistics (right figure). We can see that the DDBD algorithm is the only one for which all instances reached optimality, followed by the DBC algorithm which  proved optimality for more than $80\%$ of the instances and was unable to continue partitioning in an additional $6\%$, while the AMIO had less than $50\%$ success, and all other instances reached the time limit with an average UL gap of $0.5\%$. Moreover, the DDBD algorithm terminated after at most $150$ seconds, while for the DBC and AMIO only $45\%$ and  $10\%$ of the instances successfully terminated by that time, respectively. 

\begin{figure}[ht]\centering
\caption{Termination statistics for the location-transportation problem with $N=L=10$.}
\label{Fig:TerminateStats__LocationTransportation_N=L=10}
{\begin{flushleft} \footnotesize Graph describing the number of instances not terminated at each time (left) and the cause of termination (right).\end{flushleft}}
\includegraphics[width=.75\textwidth]{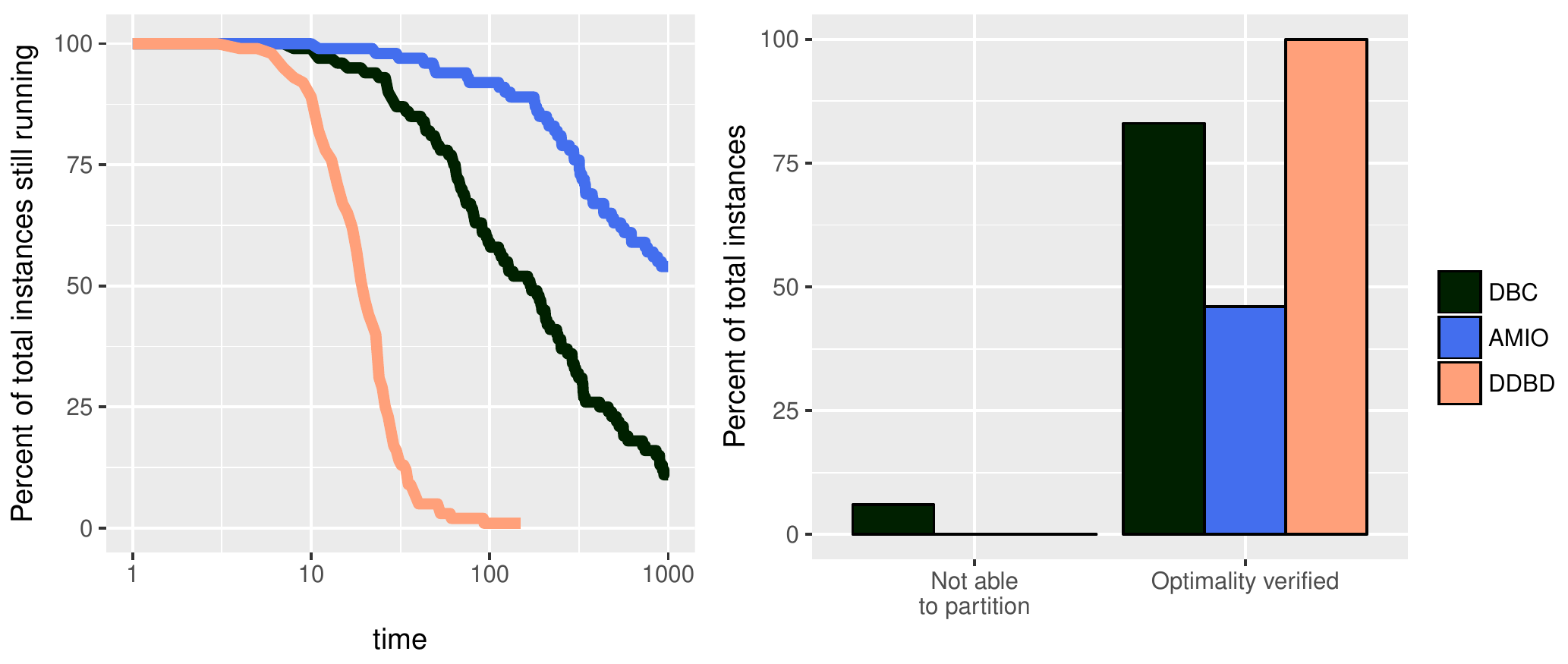}
\end{figure} 

Figure~\ref{Fig:UBLBGap_LocationTransportation_N=L=10} shows the average time to reach a certain UL gap (left) and the number of instances that reached that UL gap (right). The DDBD required $5$ seconds on average to obtain a finite UL gap, which indicates that it obtained a feasible solution. However, once the DDBD reached a finite gap it only took an average of 15 more seconds to obtain a gap of $10^{-3}$. In contrast, it took the DBC and AMIO $0.5-2$ seconds on average to obtain a feasible solution, however, they required at least an additional $100$ seconds on average to obtain a UL gap of $10^{-3}$, for the instances that reached that gap. This implies that the DDBD is much more scalable in terms of high accuracy solutions, but requires more time for low accuracy solutions.

\begin{figure}[t]
\centering
\caption{UL gap summary for the location-transportation problem with $N=L=10$.} 
\label{Fig:UBLBGap_LocationTransportation_N=L=10}
{\begin{flushleft} \footnotesize Graph describing the average time it took the instances to reach the UL gap (left) and the number of instances which reached the UL gap (right).\end{flushleft}}
\includegraphics[width=.75\textwidth]{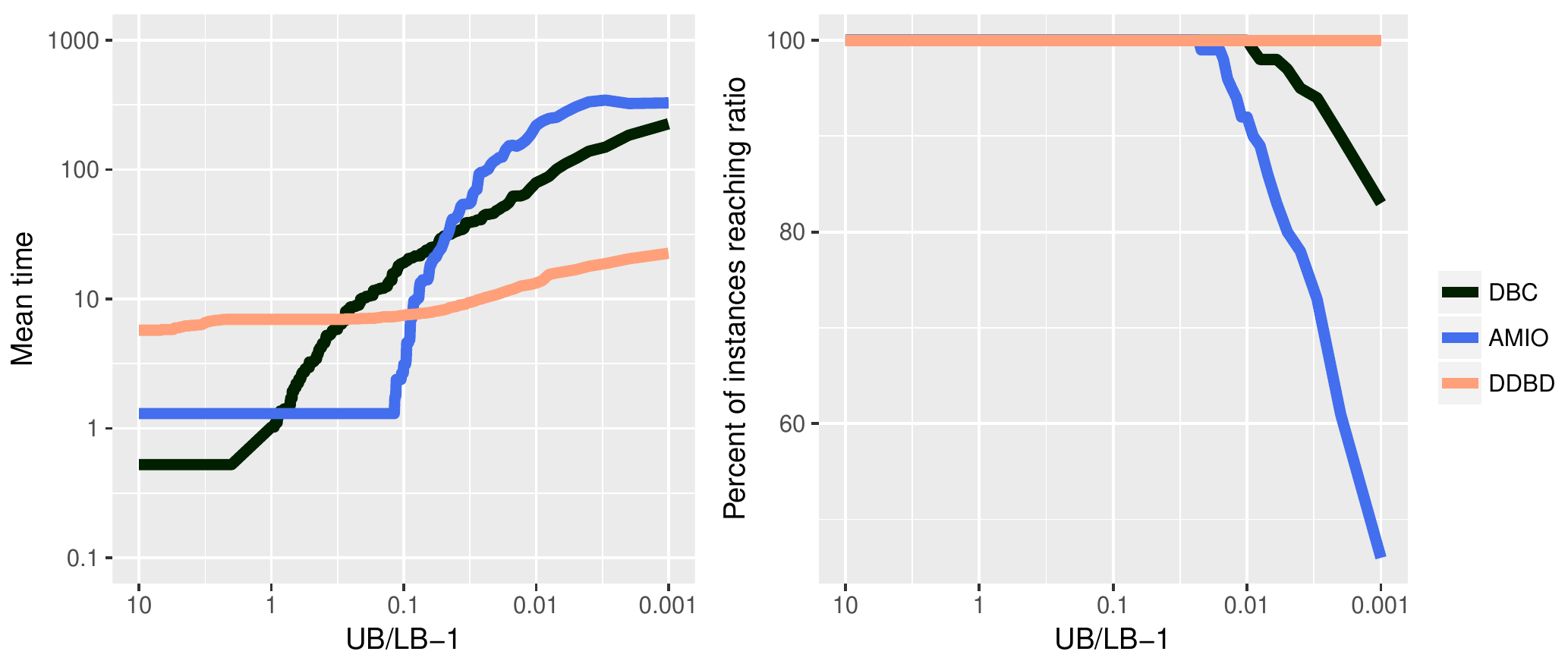}
\end{figure}

\subsection{The Capacitated Network Lot-Sizing Problem}
We consider the two-stage capacitated network lot-sizing problem. In this problem there are $N$ locations, and each location $i$ has an unknown demand $d_i$ which must be satisfied. The demand at location $i$ is satisfied by either buying stock $x_i$ in advance at cost $a_i$ per unit, or transporting amount $y_{ji}$ from location $j$, after the demand is realized, at cost $b_{ji}$ per unit. The transported amount from point $i$ to $j$ cannot exceed the capacity $c_{ij}$. We assume that the demand uncertainty set $\mathcal{D}$ is given by 
$$\mathcal{D}=\paren{\bd\in\real^N_+: d_i\leq  K,\quad\sum_i d_i\leq \sqrt{N}K}.$$
The full formulation of this two-stage problem is given in \eqref{Prob:LotSizing}.
\begin{equation}\label{Prob:LotSizing}
\begin{aligned}[t]
\min_{\bx\in\real^N_+,\by(\bd):\real^L\rightarrow\real^{N(N-1)}_+}&& \sum_{i=1}^{N} a_ix_i+\max_{\bd\in\mathcal{D}}\sum_{j\neq i} b_{ij}y_{ij}(\bd)&&\\
\qquad\text{s.t.}\qquad&& x_i-\sum_{j\neq i} y_{ij}(\bd)+\sum_{j\neq i} y_{ji}(\bd)&\geq d_i,\quad i=1,\ldots,N,\quad &\forall d\in \mathcal{D}\\
&&\by(\bd)&\leq \bc,\quad&\forall d\in \mathcal{D}\\
&&\bx&\leq K\bone.&\\
\end{aligned}
\end{equation}

We generated $100$ simulation of this problem for $N=5$, for each simulation the location of each $i$ is randomly chosen to be a point generated from a standard 2D Gaussian distribution. The cost $b_{ij}$ is set to be the Euclidean distance between point $i$ and point $j$, $a_i=1$ for all $i$, $c_{ij}=K/(N-1)u_{ij}$ where $u_{ij}$ are i.i.d. random variables generated from a standard uniform distribution, and demand is limited by $K=20$. We tested $100$ instances of the problem and set the time limit to $1000$ seconds and the required UL gap to $10^{-3}$.  

In Table~\ref{tbl:CCG_Compare_LotSizing_N=5}, we see that in $34\%$ of the instances the CCG algorithm terminated with an infeasible solution. Furthermore, for the instances in which the CCG reached optimality the DDBD took on average less time to reach optimality. However, for the DDBD algorithm, the time to verify the optimal solution after it is reached (stage (4) of Algorithm~\ref{alg:DDBD} which is not present in the CCG), was almost equal to the time it took to reach optimality. This additional runtime, which is still in the same order of magnitude as the original CCG runtime, can be viewed as the cost of ensuring feasibility.

\begin{table}[th]
\caption{Feasibility and runtime summary for the capacitated network lot-sizing problem with $N=5$.}\label{tbl:CCG_Compare_LotSizing_N=5}
\centering
{\renewcommand{\arraystretch}{0.7}
\setlength{\extrarowheight}{.1em}
\begin{tabular}{|c||c|c|c|}
\hline
Algorithm & \% Feasible & Time to reach optimality & Time to verify optimality\\
&& Mean (std) time (sec)& Mean (std) time (sec)\\
\hline
CCG & 66\% & 10.5 (2.96) & -\\
DDBD & 100\% & 7.1 (3.7) & 6.32 (12.17)\\
\hline
\end{tabular}}
\end{table}

Figure~\ref{Fig:TerminateStats_LotSizing_N=5} shows that in all the instances, the DDBD reached the optimal solution successfully after at most $100$ seconds, while the AMIO still had 75 instances running at that time, and the DBC terminated four instances without reaching optimality due to partitioning problems. The DDBC was the fastest algorithm, followed by the DBC. The AMIO had the slowest termination time, and three of the instances reached the time limit. The average UL gap for the instances that did not reach optimality was $0.66\%$ for the DBC, and $0.4\%$ for the AMIO.
\begin{figure}[h]
\caption{Termination statistics for the capacitated network lot-sizing problem with $N=5$.}
\label{Fig:TerminateStats_LotSizing_N=5}
{\begin{flushleft} \footnotesize Graph describing the number of instances not terminated at each time (left) and the cause of termination (right).\end{flushleft}}
\begin{center}\includegraphics[width=.75\textwidth]{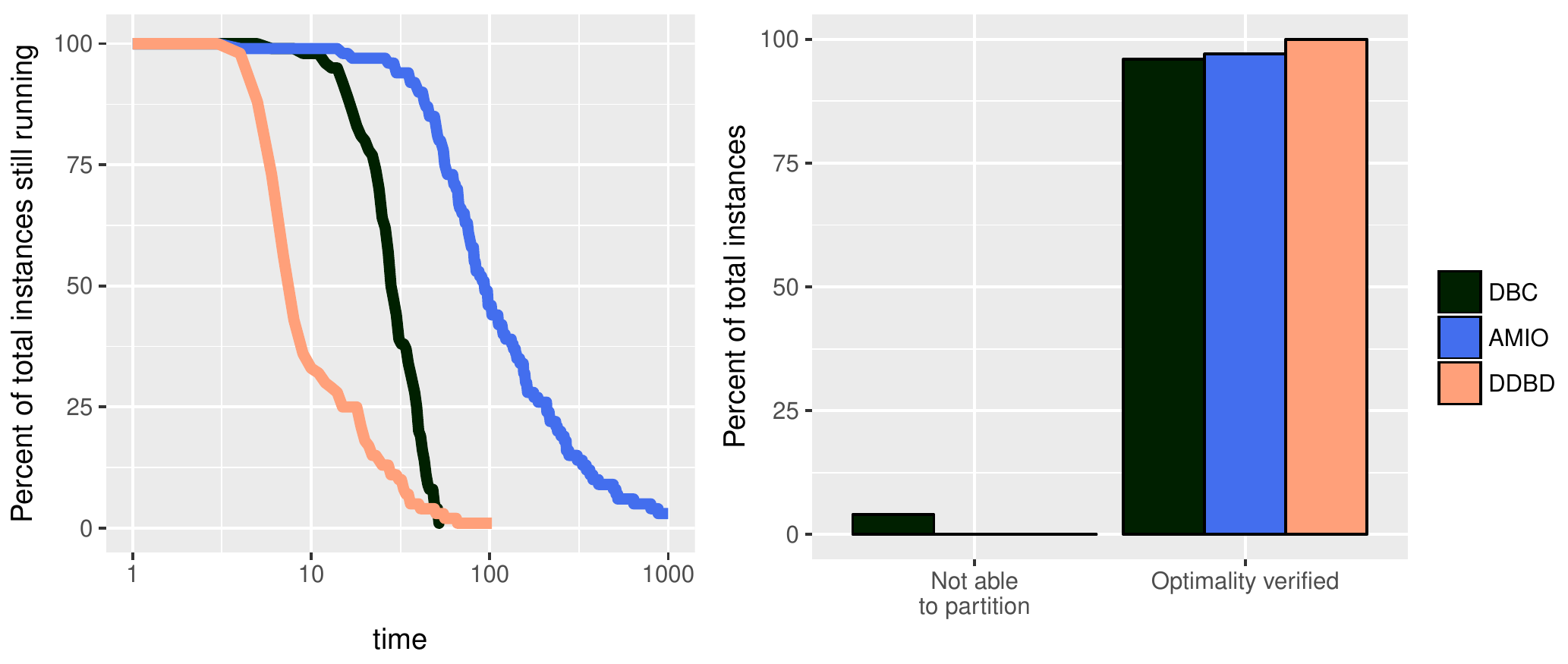}\end{center}
\end{figure}

In Figure~\ref{Fig:UBLBGapIt} we see very similar behavior to that of Figure~\ref{Fig:UBLBGap_LocationTransportation_N=L=10}. The DDBD took the longest (about 5 seconds) to initialize and obtain a feasible solution with finite UL gap, but only an additional 3 seconds more to reach the desired $10^{-3}$ UL gap. In contrast, both DBC and AMIO took less than a second to initialize but required $20$ and $100$ seconds more, respectively, to reach the desired UL gap. This fact reinforces our assertion that the DDBD is more scalable for higher accuracies.


\begin{figure}[!h]
\centering
\caption{UL gap summary for the capacitated network lot-sizing problem with $N=5$.}
\label{Fig:UBLBGapIt}
{\begin{flushleft} \footnotesize Graph describing the average time it took the instances to reach the UL gap (left) and the number of instances which reached the UL gap (right).\end{flushleft}}
\includegraphics[width=.75\textwidth]{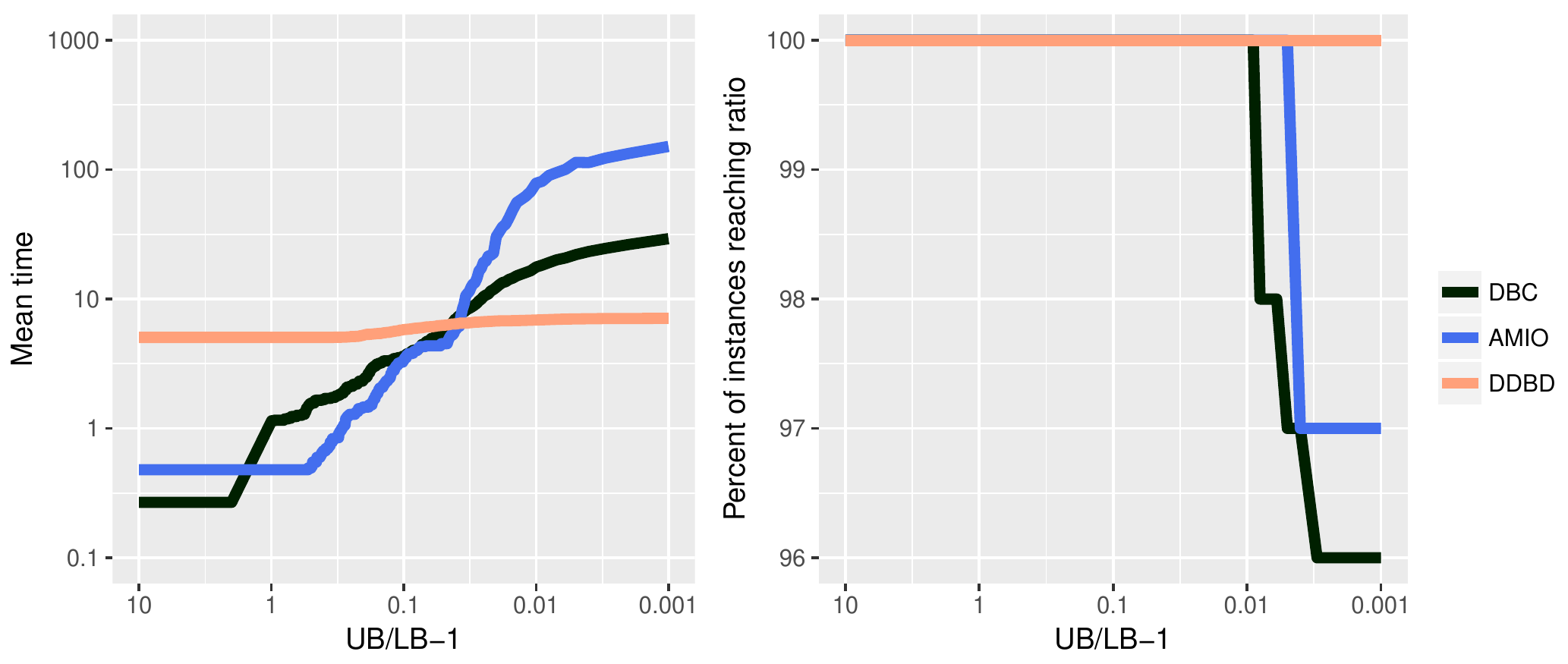}
\end{figure}

\subsection{The unit commitment problem}
The unit commitment problem is prevalent in the energy market community. In this problem, resources, such as generators, need to be assigned to either an on or an off state at various times. Usually there are constraints on the shifts between those states, as well as decisions regarding the capacity of these generators, which are all treated as design decisions (or first-stage decisions). After the demand is realized, the electricity production and its distribution among the consumers is decided upon (second-stage decisions). We used the model presented in \citep[Appendix]{Bertsimas2013}, without reserve constraints and with added variables
$z_i\in[0,1]$ such that the maximal production capacity of generator $i$ is set to $p^{\text{max}}_i=\overline{p}_i+\Delta{p}_i z_i$, and the cost of adding a unit to the maximal capacity is given by $A_i$. We used the IEEE 14-Bus test case with parameters that are given in Appendix~\ref{appx:unitcom_parameters}. 

We looked at scenarios with $T=6$, $T=12$ and $T=24$ time steps with a time limit of 1000 seconds. Both the DBC and the AMIO were not able to obtain any solution for the problem for any value of $T$.
The CCG algorithm and the DDBD algorithm both converged, however, in all cases the CCG converged to an infeasible solution (verified by the fact the the DDBD solution had higher lower bound). The results for the DDBD algorithm are given in Table~\ref{tbl:unit commitment}, where optimality time is the time it took to obtain the optimal solution, the number of inner iterations are the number of cuts added by $\mathcal{F}_1$, the verification time is the time it took to verify the solution is indeed feasible and hence optimal, and the number of outer iteration is the number of times $\mathcal{F}_2$ had to be applied. We can see that verifying feasibility of the optimal solution took much longer than finding it, due to the fact that the size of the problem becomes a major issue when solving \eqref{Prob:LinearDecision_Partition_Dual} in Step 3b of Algorithm~\ref{alg:Dcuts}. However, in all cases there was actually no need to run $\mathcal{F}_2$ to verify feasibility, since the optimal solution was indeed obtained, and no feasibility cuts were added.

The fact that the DDBD algorithm is the only one which was successfully run is an indicator of its superior scalability in the  problem size compared to the AMIO and DBC. Moreover, the ability of the $\mathcal{F}_1$ algorithm to efficiently discover the realizations which determine the optimal solution, contributed to its scalability, by making the feasibility verification, done by the more computationally expensive $\mathcal{F}_2$ algorithm, unnecessary. 
\begin{table}[h]
\centering
\caption{Unit commitment experiment results for the DDBD algorithm applied on the IEEE 14-Bus system.}
\label{tbl:unit commitment}
{\renewcommand{\arraystretch}{0.7}
\setlength{\extrarowheight}{.1em}
\begin{tabular}{|l|c|c|c|c|}
\hline
\multirow{2}{*}{$T$} & Optimality & \# Inner &Verification & \# Outer\\
& time (sec) &  Iterations & time (sec) & Iterations\\
\hline
\hline
6 & 19 & 4 &168 & 1\\
12 & 46 & 4 & 364 & 1\\
24 & 375 & 5& 12443& 1\\
\hline
\end{tabular}}
\end{table}
\section{Conclusions}
In this paper, we presented the DDBD algorithm which extends the CCG algorithm for adaptive two-stage optimization where only feasibility of the problem is assumed. We showed that by using two feasibility oracles: a fast inexact one and a slow accurate one, we can maintain the scalability of the CCG while ensuring feasibility of the solution. This scalability is manifested in the ability of the algorithm to find feasible high accuracy solutions faster then existing partitioning algorithms, and to solve large scale problems which the other methods could not practically solve. We empirically demonstrated that in some cases the fast inexact feasibility oracle, based on AM, may be enough to ensure feasibility of the solution without the need for the slower algorithm, shortening the algorithm's running time even further.

\begin{appendix}
\section{Proof of Proposition~\ref{prop:CCG_converge}}\label{appx:convergence proof}
\begin{proof}{Proof.}
First we will show that for any feasible $\bx\in X$ it must hold that $\overline{Z}(\bx)<\infty$. According to Assumption~\ref{ass:full recourse} and Proposition~\ref{prop:CCG_equiv} the equality  $\overline{Z}(\bx)=\tilde{Z}(\bx)$ holds true for any $\bx\in X$. Moreover, for every $\bu\in U$ it holds that $\underline{Z}(\bx,\{\bu\})<\infty$ and since $U$ is compact, then 
$$\overline{Z}(\bx)=\max_{\bu\in U} \underline{Z}(\bx,\{\bu\})=\underline{Z}(\bx,\{\overline{\bu}(\bx)\})<\infty.$$
Furthermore, since $V^{k-1}\subseteq V^k$, where the equality holds if and only if $\bu(\bx^k)\in V^{k-1}$.
If the algorithm stops after a finite number of steps $k$, it follows from the definition of the stopping criteria that
$$\underline{Z}(\bx^k,V^{k-1})=\underline{Z}(\bx^k,\{\overline{\bu}(\bx^k)\})= \overline{Z}(\bx^k),$$
which implies $\bx^k\in\argmin_{x\in X} \overline{Z}(\bx)$ is an optimal solution.

By Lemma~\ref{lemma:Maximizer_ExtremePoint} we have that $\overline{\bu}(
\bx)$ can always be chosen to be an extreme point of $U$, which in the case of a polytope means a vertex of $U$.
Thus, for each $\bx^k$ we add the vertex $\overline{\bu}(\bx^k)$ to $V^{k-1}$. Let $N$ be the number of vertices of the polyhedral set $U$, if the $N$th iteration is reached, then the set $V^N$ contains all vertices and specifically $\bu(\bx^{N+1})$, guaranteeing $\bx^{N+1}$ is an optimal solution, and the algorithm stops. 

If $U$ and $X$ are compact, and the algorithm does not terminate after a finite number of steps, then $V^{k-1}\subset V^k$ and  $\conv(V^{k-1})\subset \conv(V^k)\subset U$ since $\overline{\bu}(\bx^k)$ is an extreme point of the convex set $U$. Moreover, since $U$ is compact we have $\lim_{k\rightarrow\infty}\conv(V^k)=\overline{V}\subseteq U$. Following the compactness of both $X$ and $U$, there exist constants $L_1$ and $L_2$ such that for any feasible $\bx,\by\in X$ for any $V,W\subseteq U$ 
\begin{align}
|\underline{Z}(\bx,V)-\underline{Z}(\bx,W)|\leq L_1d_H(V,W),\; |\underline{Z}(\bx,V)-\underline{Z}(\by,V)|\leq L_2\norm{\bx-\by},\label{eq:Continuity}
\end{align}
where $d_H$ denotes the Hausdorff distance.
Furthermore, 
$\underline{Z}(\bx^k,V^{k-1})$ is non-decreasing, and since it is bounded from above it must converge to some value $Z$, and
$$Z=\lim_{k\rightarrow\infty}\underline{Z}(\bx^k,V^{k-1})=\lim_{k\rightarrow\infty}\min_{\bx\in X} \underline{Z}(\bx,V^{k-1})=\lim_{k\rightarrow\infty}\min_{\bx\in X} \underline{Z}(\bx,\conv(V^{k-1}))\leq \min_{x\in X} \overline{Z}(\bx),$$
where the second equality stems from the definition of $\bx^k$. Let $\bar{\bx}$ be a limit point of the sequence $\{\bx^k\}_{k\in\mathbb{N}}$ (existence is guaranteed by compactness of $X$),
then 
\begin{align*}
&\overline{Z}(\bar{\bx})-Z=\underline{Z}(\bar{\bx},U) -Z=\lim_{k\rightarrow\infty} \underline{Z}(\bx^{k},U) -  \lim_{k\rightarrow\infty}(\bx^k,\conv(V^{k-1}))=\\
&\lim_{k\rightarrow\infty} \underline{Z}(\bx^{k},\conv(V^{k}))-\underline{Z}(\bx^{k},\conv(V^{k-1}))\leq \lim_{k\rightarrow\infty} L_1 d_H(\conv(V^{k}),\conv(V^{k-1})) =0,
\end{align*}
where the third equality follows from the definition of $V^k$, the inequality follows from \eqref{eq:Continuity}, and the last equality follows from convergence of the set sequence $\{V^k\}_{k\in\mathbb{N}}$. 
Thus showing that $\bar{\bx}$ must be a minimizer of \eqref{eq:2stage_robust}.
\end{proof}

\section{Parameters for the unit commitment problem}\label{appx:unitcom_parameters}

The IEEE 14-bus example contains 14 buses ($\mathcal{N}_b=14$), 5 of which are also generators ($\mathcal{N}_g=5$), and 11 of which are loads ($\mathcal{N}_d=11$), with 20 transmission lines ($\mathcal{N}_l=20$). The demand fluctuation is set to be $5\%$ of the nominal demand. All reserve constraints are ignored. 

\begin{landscape}
\begin{table}[b]
\centering
\caption{Load demand nominal value.}
\label{tbl:Demand}
\resizebox{1\textwidth}{!}{
\begin{tabular}{|l|llllllllllllllllllllllll|}
\hline
 Loads/Time     & 1     & 2     & 3     & 4     & 5     & 6     & 7     & 8     & 9     & 10     & 11     & 12    & 13    & 14    & 15     & 16     & 17    & 18     & 19     & 20     & 21     & 22     & 23     & 24    \\
 \hline
Bus2  & 19.19 & 17.70 & 15.55 & 10.72 & 13.41 & 16.09 & 18.77 & 20.91 & 21.98 & 23.59  & 23.86  & 22.52 & 21.45 & 20.38 & 23.59  & 24.13  & 22.79 & 23.86  & 25.20  & 26.27  & 26.81  & 24.13  & 23.33  & 21.98 \\
Bus3  & 83.29 & 76.81 & 67.50 & 46.55 & 58.19 & 69.83 & 81.47 & 90.78 & 95.44 & 102.42 & 103.58 & 97.76 & 93.11 & 88.45 & 102.42 & 104.75 & 98.93 & 103.58 & 109.40 & 114.06 & 116.39 & 104.75 & 101.26 & 95.44 \\
Bus4  & 42.26 & 38.98 & 34.25 & 23.62 & 29.53 & 35.43 & 41.34 & 46.07 & 48.43 & 51.97  & 52.56  & 49.61 & 47.25 & 44.88 & 51.97  & 53.15  & 50.20 & 52.56  & 55.51  & 57.88  & 59.06  & 53.15  & 51.38  & 48.43 \\
Bus5  & 6.72  & 6.20  & 5.45  & 3.76  & 4.69  & 5.63  & 6.57  & 7.32  & 7.70  & 8.26   & 8.36   & 7.89  & 7.51  & 7.14  & 8.26   & 8.45   & 7.98  & 8.36   & 8.83   & 9.20   & 9.39   & 8.45   & 8.17   & 7.70  \\
Bus6  & 9.90  & 9.13  & 8.03  & 5.54  & 6.92  & 8.30  & 9.69  & 10.79 & 11.35 & 12.18  & 12.32  & 11.62 & 11.07 & 10.52 & 12.18  & 12.45  & 11.76 & 12.32  & 13.01  & 13.56  & 13.84  & 12.45  & 12.04  & 11.35 \\
Bus9  & 26.08 & 24.06 & 21.14 & 14.58 & 18.22 & 21.87 & 25.51 & 28.43 & 29.89 & 32.07  & 32.44  & 30.62 & 29.16 & 27.70 & 32.07  & 32.80  & 30.98 & 32.44  & 34.26  & 35.72  & 36.45  & 32.80  & 31.71  & 29.89 \\
Bus10 & 7.96  & 7.34  & 6.45  & 4.45  & 5.56  & 6.67  & 7.78  & 8.67  & 9.12  & 9.79   & 9.90   & 9.34  & 8.90  & 8.45  & 9.79   & 10.01  & 9.45  & 9.90   & 10.45  & 10.90  & 11.12  & 10.01  & 9.67   & 9.12  \\
Bus11 & 3.09  & 2.85  & 2.51  & 1.73  & 2.16  & 2.59  & 3.03  & 3.37  & 3.55  & 3.81   & 3.85   & 3.63  & 3.46  & 3.29  & 3.81   & 3.89   & 3.68  & 3.85   & 4.06   & 4.24   & 4.32   & 3.89   & 3.76   & 3.55  \\
Bus12 & 5.39  & 4.97  & 4.37  & 3.01  & 3.77  & 4.52  & 5.28  & 5.88  & 6.18  & 6.63   & 6.71   & 6.33  & 6.03  & 5.73  & 6.63   & 6.78   & 6.41  & 6.71   & 7.08   & 7.39   & 7.54   & 6.78   & 6.56   & 6.18  \\
Bus13 & 11.94 & 11.01 & 9.67  & 6.67  & 8.34  & 10.01 & 11.68 & 13.01 & 13.68 & 14.68  & 14.84  & 14.01 & 13.34 & 12.68 & 14.68  & 15.01  & 14.18 & 14.84  & 15.68  & 16.35  & 16.68  & 15.01  & 14.51  & 13.68 \\
Bus14 & 13.17 & 12.15 & 10.68 & 7.36  & 9.20  & 11.05 & 12.89 & 14.36 & 15.10 & 16.20  & 16.38  & 15.46 & 14.73 & 13.99 & 16.20  & 16.57  & 15.65 & 16.38  & 17.30  & 18.04  & 18.41  & 16.57  & 16.02  & 15.10\\
\hline
\end{tabular}}
\end{table}
\begin{table}[h]
\hfill
\begin{minipage}[t]{.25\linewidth}
\centering
\captionof{table}{Generator Data.}
\label{tbl:generator data}
\resizebox*{.8\textwidth}{!}
{
\begin{tabular}{|l|lllll|}
\hline
Generator                & G1    & G2   & G3   & G6   & G8   \\
\hline
Bus                      & Bus1  & Bus2 & Bus3 & Bus6 & Bus8 \\
\hline
$\overline{p}$ {[}MW{]}  & 166.2 & 70   & 50   & 50   & 50   \\
$\Delta{p}$ {[}MW{]}     & 332.4 & 140  & 100  & 100  & 100  \\
Pmin {[}MW{]}            & 0     & 0    & 0    & 0    & 0    \\
$RU$ {[}MW/h{]}          & 111   & 47   & 33   & 33   & 33   \\
$RD$ {[}MW{]}            & 111   & 47   & 33   & 33   & 33   \\
MinUP                    & 8     & 1    & 1    & 1    & 1    \\
MinDW                    & 8     & 1    & 1    & 1    & 1    \\
InitS                    & -8    & -1   & -1   & -1   & -1   \\
InitP                    & 0     & 0    & 0    & 0    & 0    \\
$S$                      & 8310  & 3500 & 2500 & 2500 & 2500 \\
$G$                      & 0     & 0    & 0    & 0    & 0    \\
$F$ {[}\${]}             & 1662  & 700  & 500  & 500  & 500  \\
$C$ {[}\$/MWh{]}         & 20    & 40   & 60   & 80   & 100  \\
$A$ {[}\$/MWh{]}         & 240   & 480  & 720  & 960  & 1200\\
\hline
\end{tabular}}
\end{minipage}
\hfill
\begin{minipage}[t]{.2\linewidth}
\centering
\captionof{table}{Line Data.}
\label{tbl:line data}
		\resizebox*{1\textwidth}{!}
		{
			\begin{tabular}{|l|lll|}
				\hline
				Branch Name & FromBus & ToBus & $f^{\text{max}}$ (MW) \\
				\hline
				Line1To2    & Bus1    & Bus2  & 135           \\
				Line1To5    & Bus1    & Bus5  & 135           \\
				Line2To3    & Bus2    & Bus3  & 135           \\
				Line2To4    & Bus2    & Bus4  & 135           \\
				Line2To5    & Bus2    & Bus5  & 135           \\
				Line3To4    & Bus3    & Bus4  & 135           \\
				Line4To5    & Bus4    & Bus5  & 135           \\
				Line4To7    & Bus4    & Bus7  & 135           \\
				Line4To9    & Bus4    & Bus9  & 135           \\
				Line5To6    & Bus5    & Bus6  & 135           \\
				Line6To11   & Bus6    & Bus11 & 135           \\
				Line6To12   & Bus6    & Bus12 & 135           \\
				Line6To13   & Bus6    & Bus13 & 135           \\
				Line7To8    & Bus7    & Bus8  & 135           \\
				Line7To9    & Bus7    & Bus9  & 135           \\
				Line9To10   & Bus9    & Bus10 & 135           \\
				Line9To14   & Bus9    & Bus14 & 135           \\
				Line10To11  & Bus10   & Bus11 & 135           \\
				Line12To13  & Bus12   & Bus13 & 135           \\
				Line13To14  & Bus13   & Bus14 & 135          \\
				\hline
		\end{tabular}}
\end{minipage}
\hfill
\begin{minipage}[t]{.525\linewidth}
\caption{Shift factors.}
\label{tbl:ShiftFactors}
\centering
\resizebox{1\textwidth}{!}{
\begin{tabular}{|l|lllllllllllll|}
     \hline
   Line        & Bus2         & Bus3         & Bus4         & Bus5         & Bus6         & Bus7         & Bus8         & Bus9         & Bus10        & Bus11        & Bus12        & Bus13        & Bus14        \\
           \hline
Line1To2   		&	-0.8380	&	-0.7465	&	-0.6675	&	-0.6106	&	-0.6291	&	-0.6573	&	-0.6573	&	-0.6518	&	-0.6477	&	-0.6386	&	-0.6309	&	-0.6323	&	-0.6433	\\
Line1To5   		&	-0.1620	&	-0.2535	&	-0.3325	&	-0.3894	&	-0.3709	&	-0.3427	&	-0.3427	&	-0.3482	&	-0.3523	&	-0.3614	&	-0.3691	&	-0.3677	&	-0.3567	\\
Line2To3   		&	0.0273	&	-0.5320	&	-0.1513	&	-0.1031	&	-0.1188	&	-0.1427	&	-0.1427	&	-0.1380	&	-0.1346	&	-0.1269	&	-0.1204	&	-0.1215	&	-0.1308	\\
Line2To4   		&	0.0572	&	-0.1434	&	-0.3167	&	-0.2158	&	-0.2487	&	-0.2986	&	-0.2986	&	-0.2888	&	-0.2817	&	-0.2655	&	-0.2519	&	-0.2543	&	-0.2738	\\
Line2To5   		&	0.0774	&	-0.0711	&	-0.1994	&	-0.2917	&	-0.2616	&	-0.2160	&	-0.2160	&	-0.2249	&	-0.2314	&	-0.2463	&	-0.2587	&	-0.2564	&	-0.2387	\\
Line3To4   		&	0.0273	&	0.4680	&	-0.1513	&	-0.1031	&	-0.1188	&	-0.1427	&	-0.1427	&	-0.1380	&	-0.1346	&	-0.1269	&	-0.1204	&	-0.1215	&	-0.1308	\\
Line4To5   		&	0.0799	&	0.3067	&	0.5026	&	-0.3012	&	-0.0389	&	0.3584	&	0.3584	&	0.2808	&	0.2240	&	0.0948	&	-0.0137	&	0.0061	&	0.1607	\\
Line4To7   		&	0.0030	&	0.0113	&	0.0186	&	-0.0111	&	-0.2075	&	-0.6338	&	-0.6338	&	-0.4469	&	-0.4043	&	-0.3076	&	-0.2264	&	-0.2412	&	-0.3569	\\
Line4To9   		&	0.0017	&	0.0066	&	0.0108	&	-0.0065	&	-0.1211	&	-0.1658	&	-0.1658	&	-0.2608	&	-0.2360	&	-0.1795	&	-0.1321	&	-0.1408	&	-0.2083	\\
Line5To6   		&	-0.0047	&	-0.0179	&	-0.0294	&	0.0176	&	-0.6714	&	-0.2004	&	-0.2004	&	-0.2924	&	-0.3597	&	-0.5128	&	-0.6415	&	-0.6181	&	-0.4348	\\
Line6To11  		&	-0.0028	&	-0.0108	&	-0.0177	&	0.0106	&	0.1979	&	-0.1207	&	-0.1207	&	-0.1760	&	-0.2873	&	-0.5402	&	0.1683	&	0.1452	&	-0.0356	\\
Line6To12  		&	-0.0004	&	-0.0016	&	-0.0026	&	0.0016	&	0.0291	&	-0.0177	&	-0.0177	&	-0.0259	&	-0.0161	&	0.0061	&	-0.5211	&	-0.1697	&	-0.0887	\\
Line6To13  		&	-0.0014	&	-0.0056	&	-0.0091	&	0.0055	&	0.1017	&	-0.0620	&	-0.0620	&	-0.0904	&	-0.0563	&	0.0213	&	-0.2886	&	-0.5936	&	-0.3104	\\
Line7To8   		&	0.0000	&	0.0000	&	0.0000	&	0.0000	&	0.0000	&	0.0000	&	-1.0000	&	0.0000	&	0.0000	&	0.0000	&	0.0000	&	0.0000	&	0.0000	\\
Line7To9   		&	0.0030	&	0.0113	&	0.0186	&	-0.0111	&	-0.2075	&	0.3662	&	0.3662	&	-0.4469	&	-0.4043	&	-0.3076	&	-0.2264	&	-0.2412	&	-0.3569	\\
Line9To10  		&	0.0028	&	0.0108	&	0.0177	&	-0.0106	&	-0.1979	&	0.1207	&	0.1207	&	0.1760	&	-0.7127	&	-0.4598	&	-0.1683	&	-0.1452	&	0.0356	\\
Line9To14  		&	0.0019	&	0.0071	&	0.0117	&	-0.0070	&	-0.1307	&	0.0797	&	0.0797	&	0.1163	&	0.0724	&	-0.0274	&	-0.1902	&	-0.2367	&	-0.6008	\\
Line10To11 		&	0.0028	&	0.0108	&	0.0177	&	-0.0106	&	-0.1979	&	0.1207	&	0.1207	&	0.1760	&	0.2873	&	-0.4598	&	-0.1683	&	-0.1452	&	0.0356	\\
Line12To13 		&	-0.0004	&	-0.0016	&	-0.0026	&	0.0016	&	0.0291	&	-0.0177	&	-0.0177	&	-0.0259	&	-0.0161	&	0.0061	&	0.4789	&	-0.1697	&	-0.0887	\\
Line13To14 		&	-0.0019	&	-0.0071	&	-0.0117	&	0.0070	&	0.1307	&	-0.0797	&	-0.0797	&	-0.1163	&	-0.0724	&	0.0274	&	0.1902	&	0.2367	&	-0.3992	\\
     \hline
\end{tabular}}
\hfill
\end{minipage}
\end{table}
\end{landscape}
\end{appendix}


\bibliographystyle{informs2014} 
\bibliography{CutsBib} 


\end{document}